\documentclass[10pt]{article}
\usepackage{amsfonts,amssymb,amsmath,amsthm}
\usepackage{cite}
\usepackage{needspace}
\usepackage[dvipsnames]{xcolor}
\usepackage{pgfplots}
\pgfplotsset{compat=1.18}

\definecolor{gapBoundary}{RGB}{0,114,178}
\definecolor{gapThreshold}{RGB}{213,94,0}
\definecolor{gapStable}{RGB}{0,135,100}

\usepackage[
    colorlinks=true,
    linkcolor=black,
    urlcolor=blue,
    citecolor=ForestGreen
]{hyperref}
\hypersetup{
    pdftitle={The principal spectral gap for birth--death processes with a strong Allee effect},
    pdfauthor={Dun Zhou}
}
\renewcommand{\eqref}[1]{%
    \textup{(\hyperref[#1]{\textcolor{red}{\ref*{#1}}})}%
}
\DeclareRobustCommand{\resultref}[2]{%
    #1\hyperref[#2]{\textcolor{red}{\ref*{#2}}}%
}
\newtheorem{theorem}{Theorem}[section]
\newtheorem{lemma}[theorem]{Lemma}
\newtheorem{proposition}[theorem]{Proposition}

\numberwithin{equation}{section}
\newtheorem{mainthm}{Theorem}

\newtheorem{maincor}[mainthm]{Corollary}

\begin{document}
\setlength{\baselineskip}{13pt}

\title{The principal spectral gap for birth--death processes
with a strong Allee effect}

\author{Dun Zhou\thanks{Email: zhoudun@njust.edu.cn. Supported by the National Natural Science Foundation of China
(Grant Nos.~12671213 and~12331006).}\\
School of Mathematics and Statistics\\
Nanjing University of Science and Technology\\
Nanjing 210094, China\\}
\date{}

\maketitle

\begin{abstract}
For birth--death processes with a strong Allee effect, long survival near
a positive stable equilibrium does not by itself determine the principal
spectral gap. We identify its large-population limit when the birth and
death rates at population size $n$ are $n\widetilde\lambda(n/K)$ and
$n\widetilde\mu(n/K)$, where $K$ is the population scale and
$\widetilde\lambda,\widetilde\mu$ are the per-capita rates.
Write $V(x)=x(\widetilde\lambda(x)-\widetilde\mu(x))$ for the
deterministic drift, with unstable threshold $x_1$ and positive stable
equilibrium $x_2$. For the two smallest eigenvalues
$\rho_1(K)<\rho_2(K)$ of the negative generator killed at extinction,
we prove
\[
\lim_{K\to\infty}\bigl(\rho_2(K)-\rho_1(K)\bigr)
=\min\{\widetilde\mu(0)-\widetilde\lambda(0),V'(x_1),-V'(x_2)\},
\]
so the limit depends on the linearization rates at all three equilibria.
Each term can be the unique minimum even when $x_1$ and $x_2$ are fixed.
The proof combines a frozen boundary model with two local oscillator
limits through discrete Ismagilov--Morgan--Simon (IMS) localization.
A uniform comparison of the global and stable local ground states
controls the orthogonality constraint in the variational lower bound;
projected local trial vectors give the matching upper bound.
The result determines the limiting $L^2$ spectral gap and optimal
Poincar\'e constant of the $Q$-process, which describes conditioning on
indefinite survival. Combined with principal-eigenvalue asymptotics
from a companion work, it separates an exponentially growing
quasi-stationary mean extinction time from a spectral relaxation time
with a finite positive limit.
\end{abstract}



\tableofcontents
\section{Introduction}
\label{sec:introduction}
\label{subsec:motivation}
Birth--death processes describe population dynamics through individual
births and deaths. Randomness in these events can lead to extinction
even when the deterministic model predicts persistence. Ovaskainen
and Meerson \cite[Boxes~2 and~3]{Ovaskainen-Meerson2010} review how
the master equation, diffusion approximation, and WKB method describe
such fluctuations and estimate extinction times. For the models
studied here, absorption at state $0$ occurs almost surely from every
positive state, so the only stationary distribution of the full
process is the point mass $\delta_0$. To describe the population
before extinction, one instead studies the process conditioned on
survival and the spectrum of its killed generator.

Karlin and McGregor \cite{Karlin-McGregor1957} express the transition
probabilities of birth--death processes through orthogonal polynomials
and a spectral measure. Using this representation and a duality
construction, van Doorn \cite[Sections~3--5]{vanDoorn1991}
characterizes the quasi-stationary distributions, identifies the
conditional limit for finitely supported initial laws, and determines
the exponential rate at which the conditional transition probabilities
approach their limits. Collet et~al.\ \cite[Sections~5.2--5.5 and~5.7]{Collet-Martinez-SanMartin2013}
give a systematic account of these spectral methods, the role of the
boundary at infinity, and conditions ensuring a discrete spectrum.

A quasi-stationary distribution is a probability law on the positive
states that is invariant under conditioning on survival
\cite[Definition~3]{Meleard-Villemonais2012}. The extinction time
under this law is exponential \cite[Proposition~2]{Meleard-Villemonais2012};
in our notation, its survival probability is $e^{-\rho_1(K)t}$ and
its mean is $\rho_1(K)^{-1}$. Relaxation among surviving trajectories
is described by the $Q$-process. Champagnat and Villemonais
\cite[Theorems~2.1 and~3.1]{Champagnat-Villemonais2016} give criteria
for uniform exponential convergence to a unique quasi-stationary
distribution and show that these criteria also yield the $Q$-process
as a limit of conditioning on survival to increasingly distant times.
Their Theorem~3.1(ii) identifies its transition semigroup as a Doob
transform by a positive eigenfunction. In the reversible setting
considered here, the negative generator of the $Q$-process has
eigenvalues $\rho_j(K)-\rho_1(K)$, $j\ge1$.
Consequently, $\rho_2(K)-\rho_1(K)$ is the $L^2$ spectral gap of the
$Q$-process and gives its optimal exponential relaxation rate. We call
this difference the \emph{principal spectral gap} associated with
the killed generator.

We consider populations with a strong Allee effect. The biological
motivation comes from the benefits of aggregation studied by Allee
\cite{Allee1931} and the mechanisms of positive density dependence
at low abundance reviewed by Courchamp et~al.\
\cite{Courchamp-CluttonBrock-Grenfell1999}. Stephens et~al.\
\cite[p.~186]{Stephens-Sutherland-Freckleton1999} distinguish effects
on individual fitness components from effects on population growth;
the former need not imply the latter. Here the Allee effect acts on
the population growth rate and is strong: the per-capita growth rate
is negative below a positive threshold and positive immediately above
it, as in \cite[Chapter~1]{Courchamp-Berec-Gascoigne2008}.
In the deterministic limit, the attracting equilibria $0$ and $x_2>0$
are separated by an unstable threshold $x_1\in(0,x_2)$. Trajectories
below $x_1$ decline towards extinction, whereas those above it
approach $x_2$. For stochastic diffusion models, Dennis
\cite[pp.~392--393]{Dennis2002} relates such thresholds to the
probability of reaching a lower population level before an upper one,
showing how fluctuations modify the deterministic distinction between
decline and recovery.

In the discrete model, demographic fluctuations can carry a population
near $Kx_2$ across the threshold and eventually to extinction.
Assaf and Meerson \cite[Section~V]{Assaf-Meerson2010} study this
extinction scenario by matching WKB approximations of the master
equation. They obtain the quasi-stationary profile and the exponential
term and prefactor of the mean extinction time; their Eq.~(95)
gives the formula for single-step processes. For the class of
birth--death processes considered here, Hou et~al.\
\cite[Theorem~A]{Hou-Yan-Zhou2026} prove sharp asymptotics for the
principal eigenvalue with a quantitative error bound and estimates
on its eigenfunction. Their Theorem~D describes intermediate-time
population laws as a mixture of extinction and quasi-stationarity,
with weights depending on the initial state.

The remaining issue is to identify which part of this bistable
geometry determines the $L^2$ relaxation rate of the $Q$-process.
The normalized ground state of the associated Jacobi
operator concentrates in $\ell^2$ near $Kx_2$, but the gap can be
determined by the absorbing boundary or the unstable threshold.
We prove that its limit is the minimum of the three local spectral
rates associated with $0$, $x_1$, and $x_2$. All three terms are
necessary for this class of models: each can be the unique minimum
even when the two positive equilibria are kept fixed.

The proof uses discrete IMS localization to combine the three local
estimates. Its key point is a uniform control of the stable local
ground-state component for vectors orthogonal to the global ground
state. This allows a direct proof of the gap limit without first
establishing convergence of the full spectrum. The sharp
principal-eigenvalue asymptotics from \cite{Hou-Yan-Zhou2026} enter
only in the subsequent comparison of extinction and relaxation times
in \resultref{Corollary~}{cor:two-scale-spectrum}.

\subsection{Model and assumptions}
\label{subsec:model}
For each $K\geq1$, let
\[
    X^K=(X_t^K)_{t\geq0}
\]
be a continuous-time birth--death process on
$\mathbb N=\{0,1,2,\ldots\}$, with $0$ absorbing. On
$\mathbb N^*=\{1,2,\ldots\}$, its birth and death rates are
\[
    \lambda_n^{(K)}
    =
    n\widetilde\lambda\left(\frac nK\right),
    \qquad
    \mu_n^{(K)}
    =
    n\widetilde\mu\left(\frac nK\right),
    \qquad n\geq1.
\]
Here $K$ is the population scaling parameter, and
$\widetilde\lambda,\widetilde\mu$ denote positive per-capita rate
functions defined on $\mathbb R_+=[0,\infty)$.
We write $\tau_0:=\inf\{t\ge0:X_t^K=0\}$ for the extinction time;
its dependence on $K$ is suppressed. The probability and expectation
for initial state $n$ are denoted by $\mathbb P_n$ and $\mathbb E_n$,
with an initial law in place of $n$ when appropriate.
We use the rate hypotheses {\rm(A1)}--{\rm(A3)} and the potential
condition {\rm(H)} from the companion work \cite{Hou-Yan-Zhou2026}.

\begin{enumerate}
\item[{\rm(A1)}]
The functions $\widetilde\lambda$ and $\widetilde\mu$ belong to
$C^2(\mathbb R_+)$, and
\[
    0<\widetilde\lambda(0)<\widetilde\mu(0).
\]
\item[{\rm(A2)}]
Both $\widetilde\lambda$ and $\widetilde\mu$ are increasing.
The two rate functions intersect at exactly two points $0<x_1<x_2$:
\[
    \widetilde\lambda(x_i)=\widetilde\mu(x_i),
    \qquad i=1,2.
\]
Both intersections are transverse, with
\[
    \widetilde\lambda'(x_1)>\widetilde\mu'(x_1),
    \qquad
    \widetilde\lambda'(x_2)<\widetilde\mu'(x_2).
\]
Moreover,
\[
\begin{cases}
    \widetilde\lambda(x)<\widetilde\mu(x),
        &x\in(0,x_1),\\[1mm]
    \widetilde\lambda(x)>\widetilde\mu(x),
        &x\in(x_1,x_2),\\[1mm]
    \widetilde\lambda(x)<\widetilde\mu(x),
        &x\in(x_2,\infty),
\end{cases}
\]
and the function
\[
    x\longmapsto
    \ln\frac{\widetilde\mu(x)}{\widetilde\lambda(x)}
\]
is strictly decreasing on $(0,x_1)$.
\item[{\rm(A3)}]
The rate functions satisfy the following conditions at infinity:
\[
    \lim_{x\to\infty}
    \frac{\widetilde\lambda(x)}{\widetilde\mu(x)}
    =0,
    \qquad
    \sup_{x\in\mathbb R_+}
    \frac{\widetilde\mu'(x)}{\widetilde\mu(x)}
    <\infty,
\]
and
\[
    \int_{x_2}^{\infty}
    \frac{dx}{x\widetilde\mu(x)}
    <\infty.
\]
\end{enumerate}

Define the reversible weights by
\[
    \pi_1^{(K)}:=\frac{1}{\mu_1^{(K)}}
\]
and
\[
    \pi_n^{(K)}
    :=
    \frac{
        \lambda_1^{(K)}\cdots\lambda_{n-1}^{(K)}
    }{
        \mu_1^{(K)}\cdots\mu_n^{(K)}
    },
    \qquad n\geq2.
\]
They satisfy the detailed-balance identity
\[
    \lambda_n^{(K)}\pi_n^{(K)}
    =
    \mu_{n+1}^{(K)}\pi_{n+1}^{(K)},
    \qquad n\geq1.
\]

Assumption {\rm(A3)} also implies that, for every fixed $K\ge1$,
\begin{equation}\label{eq:reversible-weight-decay}
    \lim_{n\to\infty}\bigl(\mu_n^{(K)}\bigr)^2\pi_n^{(K)}=0.
\end{equation}
The proof is given in Subsection~\ref{subsec:self-adjoint-realization}.

Define the potential
\[
    H(x)
    :=
    \int_{x_2}^{x}
    \ln\frac{\widetilde\mu(s)}{\widetilde\lambda(s)}\,ds,
    \qquad x\in\mathbb R_+.
\]
We also assume
\begin{equation*}
\tag{H}\label{assumption-H}
    H\in C^3(\mathbb R_+),
    \qquad
    \sup_{x\in\mathbb R_+}
    (1+x^2)|H'''(x)|<\infty.
\end{equation*}

The corresponding deterministic ordinary differential equation is
\[
    \dot x=V(x),
    \qquad
    V(x)
    :=
    x\bigl(
        \widetilde\lambda(x)-\widetilde\mu(x)
    \bigr).
\]
Under {\rm(A1)}--{\rm(A2)}, $0$ and $x_2$ are locally
asymptotically stable, whereas $x_1$ is unstable. In particular,
\[
    V'(0)
    =
    \widetilde\lambda(0)-\widetilde\mu(0)<0,
    \qquad
    V'(x_1)>0,
    \qquad
    V'(x_2)<0.
\]
Furthermore,
\[
    H'(x_i)=0,
    \qquad
    H''(x_i)
    =
    -\frac{V'(x_i)}{x_i\widetilde\lambda(x_i)},
    \qquad i=1,2.
\]
Thus,
\[
    H''(x_1)<0,
    \qquad
    H''(x_2)>0,
\]
so $H$ has a non-degenerate local maximum at $x_1$ and a
non-degenerate local minimum at $x_2$, with
$H(x_2)=0$.

In the analysis below, hypotheses {\rm(A1)}--{\rm(A2)} specify the behavior near the extinction boundary and the non-degeneracy of the two interior equilibria.
{\rm(A3)} gives the tail bounds used to prove compactness of the resolvent. The local harmonic-oscillator analysis requires only $C^2$ regularity of the per-capita rates near $x_1$ and $x_2$. The weighted derivative bound in {\rm(H)} is used through the sharp principal-eigenvalue asymptotics in
\resultref{Corollary~}{cor:two-scale-spectrum}.
All main results are stated under {\rm(A1)}--{\rm(A3)} and {\rm(H)},
which also suffice for the principal-eigenvalue and quasi-stationary
results of \cite{Hou-Yan-Zhou2026}.
The spectral-gap proof itself does not use {\rm(H)}.

\subsection{Main results}
\label{subsec:main-results}
The generator killed upon hitting $0$ is given by
\[
    (L_Kf)_n
    =
    \lambda_n^{(K)}(f_{n+1}-f_n)
    +
    \mu_n^{(K)}(f_{n-1}-f_n),
    \qquad n\geq1,
\]
with the Dirichlet boundary condition $f_0=0$. Under our standing assumptions,
$L_K$ has a self-adjoint realization on $\ell^2(\pi^{(K)})$
with compact resolvent; its construction and compactness argument appear in Subsection~\ref{subsec:self-adjoint-realization}. We denote
the eigenvalues of $-L_K$ by
\begin{equation}
\label{eq:killed-spectrum}
    0<\rho_1(K)<\rho_2(K)<\rho_3(K)<\cdots.
\end{equation}
A local trial state constructed in
\resultref{Lemma~}{lem:x2-window-kernel-gap}(i) gives
\begin{equation}
\label{eq:rho1-vanishes}
    \rho_1(K)\longrightarrow0
    \qquad\text{as }K\to\infty.
\end{equation}
We quote the sharp asymptotic formula from
\cite[Theorem~A]{Hou-Yan-Zhou2026} in
\resultref{Corollary~}{cor:two-scale-spectrum}.

Define the three local rates by
\begin{equation}
\label{eq:local-rates}
\begin{aligned}
    \gamma_0
    &:={}
    \widetilde\mu(0)-\widetilde\lambda(0)
    =-V'(0),\\
    \gamma_1
    &:={}
    V'(x_1)
    =-x_1\widetilde\lambda(x_1)H''(x_1),\\
    \gamma_2
    &:={}
    -V'(x_2)
    =x_2\widetilde\lambda(x_2)H''(x_2).
\end{aligned}
\end{equation}
All three constants are strictly positive. Set
\[
    \gamma_*:=\min\{\gamma_0,\gamma_1,\gamma_2\}.
\]

The main result is the following.
\begin{mainthm}[Limit of the principal spectral gap]
\label{thm:spectral-gap-limit}
Suppose that {\rm(A1)}--{\rm(A3)} and {\rm(H)} hold. Then
\[
    \lim_{K\to\infty}
    \bigl(\rho_2(K)-\rho_1(K)\bigr)
    =
    \gamma_*
    =
    \min\left\{
        \widetilde\mu(0)-\widetilde\lambda(0),
        V'(x_1),
        -V'(x_2)
    \right\}.
\]
\end{mainthm}

The three entries in the minimum have distinct spectral origins:
$\gamma_0$ is the lowest spectral value of the frozen absorbing
boundary model, $\gamma_1$ is the lowest eigenvalue of the shifted
oscillator at the unstable threshold, and $\gamma_2$ is the first
excited eigenvalue of the oscillator at the stable equilibrium.
In particular, the positive instability rate $V'(x_1)$ can determine
the spectral gap of the conditioned process.

Thus the limiting gap depends on the magnitudes of the drift
derivatives at all three equilibria, although the Jacobi ground state
concentrates near $x_2$. No ordering of these three rates is assumed.
\resultref{Proposition~}{prop:realization-three-mechanisms} shows that,
for any fixed $0<x_1<x_2$, each rate can be the unique minimum within
the admissible class. Consequently, none of the three terms can be
omitted from the formula in general.

For the corresponding conditioned process, let
$\varphi_{1,K}$ denote a positive principal eigenfunction of $-L_K$,
and write
\[
Z_K:=\|\varphi_{1,K}\|_{\ell^2(\pi^{(K)})}^2
=\sum_{n\ge1}\pi_n^{(K)}\varphi_{1,K}(n)^2.
\]
The generator of the associated $Q$-process is
\[
    \mathcal L_K^Qf
    =
    \frac{1}{\varphi_{1,K}}
    L_K(\varphi_{1,K}f)
    +
    \rho_1(K)f,
\]
and its invariant probability measure is
\[
    m_n^{(K)}
    =
    \frac{
        \pi_n^{(K)}\varphi_{1,K}(n)^2
    }{Z_K}.
\]
Let $(P_t^{Q,K})_{t\geq0}$ denote the semigroup of the
$Q$-process.

\begin{maincor}[$Q$-process relaxation]
\label{cor:q-process-gap}
Under {\rm(A1)}--{\rm(A3)} and {\rm(H)}, the
$L^2(m^{(K)})$-spectral gap of the $Q$-process equals
\[
    \rho_2(K)-\rho_1(K).
\]
Consequently,
\[
    \operatorname{gap}(-\mathcal L_K^Q)\longrightarrow\gamma_*,
\]
and the optimal Poincar\'e constant converges to $1/\gamma_*$.
\end{maincor}

The spectral shift under the Doob transform gives the gap identity;
\resultref{Theorem~}{thm:spectral-gap-limit} supplies its explicit
large-population limit. Subsection~\ref{subsec:q-process-consequences}
gives the proof and the associated Poincar\'e inequality and semigroup
decay estimate. Combining the gap limit with the sharp
principal-eigenvalue asymptotics from \cite[Theorem~A]{Hou-Yan-Zhou2026}
also determines the two distinct time scales below.

\begin{maincor}[Separation of extinction and relaxation scales]
\label{cor:two-scale-spectrum}
Under {\rm(A1)}--{\rm(A3)} and {\rm(H)},
\begin{equation}
\label{eq:principal-eigenvalue-asymptotics}
    \rho_1(K)
    =
    A e^{-KH(x_1)}
    \left(
        1+
        O\left(
            \frac{(\ln K)^3}{\sqrt K}
        \right)
    \right),
\end{equation}
where
\[
    A
    =
    \frac{
        x_2\widetilde\lambda(x_2)
        \sqrt{-H''(x_1)H''(x_2)}
    }{2\pi},
\]
whereas
\[
    \rho_2(K)=\gamma_*+o(1).
\]
In particular,
\[
    \frac{\rho_2(K)-\rho_1(K)}{\rho_1(K)}
    \longrightarrow\infty.
\]
Accordingly, the inverse $L^2$-spectral gap of the $Q$-process
converges to $1/\gamma_*$, while the quasi-stationary mean extinction time behaves asymptotically as
$A^{-1}e^{KH(x_1)}$.
\end{maincor}

\subsection{Relation to earlier results and proof strategy}
\label{subsec:comparison}
\label{subsec:proof-strategy}

Several methods relate the spectral gap of a birth--death chain to
its transition rates. Miclo \cite[Proposition~2]{Miclo1999} uses
weighted discrete Hardy inequalities to obtain upper and lower bounds,
within universal factors, from the reversible tail masses and inverse
edge weights. Chen and Saloff-Coste
\cite{Chen-SaloffCoste2014} develop convergent iterative
and bisection procedures for computing the spectral gap on finite
paths. For more general reversible Markov chains, Bovier et~al.\
\cite{Bovier-Eckhoff-Gayrard-Klein2002} relate small eigenvalues to
inverse mean metastable exit times under suitable separation and
nondegeneracy assumptions. The question here is to determine the
exact positive limit of the gap as the population scale tends to
infinity, when the principal eigenvalue tends to zero.

The closest results in this scaling regime are due to Chazottes
et~al.\ For birth--death models with a repelling origin and a single
positive attracting equilibrium, they obtain a sharp asymptotic
formula for the principal eigenvalue
\cite[Theorem~3.2]{Chazottes-Collet-Meleard2015}. They subsequently
prove that every fixed eigenvalue converges to the ordered union of
the boundary and stable-equilibrium spectra, with repeated levels
counted according to their multiplicities
\cite[Theorem~1.2]{Chazottes-Collet-Meleard2023}.
Their Corollary~1.3 identifies the gap limit as the smaller of the
boundary instability rate and the stable-equilibrium relaxation rate,
and Proposition~1.4 relates this gap to relaxation of the $Q$-process.

In the present bistable setting, the origin is attracting and the
additional unstable equilibrium contributes a third local spectrum.
The relevant levels are the boundary spectral bottom $\gamma_0$,
the lowest threshold level $\gamma_1$, and the first positive stable
oscillator level $\gamma_2$. The boundary spectrum follows from
\cite[Theorem~7.1]{Chazottes-Collet-Meleard2023} after exchanging
the birth and death constants; the details are given in
Appendix~\ref{sec:frozen-boundary-spectrum}. Our main theorem shows
how these three contributions determine the global gap. Moreover,
\resultref{Proposition~}{prop:realization-three-mechanisms} shows
that each can be the unique minimum while the two positive equilibria
remain fixed. Thus no one of the three terms can be omitted for this
class of models.

The proof treats the second variational eigenvalue directly.
Conjugation by the square roots of the reversible weights transforms
$-L_K$ into a self-adjoint Jacobi operator $\mathcal H_K$ with
quadratic form $Q_K$. The variational constraint for $\rho_2(K)$
requires unit vectors to be orthogonal to the normalized global
ground state $v_{1,K}$. To estimate their energy, we use the IMS
partition-of-unity formula for quadratic forms
\cite{Cycon-Froese-Kirsch-Simon1987}; its version for difference
operators is given by Klein and Rosenberger
\cite[Lemma~3.3(a), Eq.~(3.15)]{Klein-Rosenberger2009}.
The partition separates the absorbing boundary, the two interior
equilibria, the intermediate region, and the right tail. Interior
cutoffs of width $K^{3/4}$ contain the fluctuation scale $\sqrt K$
and give a total localization error of $O_\delta(K^{-1/2})$ for
each fixed truncation parameter $\delta$.

At the stable equilibrium, the limiting oscillator has a zero mode,
and multiplication by a cutoff need not preserve orthogonality to
the corresponding local ground state. We first apply IMS to the
global ground state itself and compare it with the stable local
ground state. Estimate~\eqref{eq:x2-uniform-orthogonality} then
controls the local zero-mode component uniformly for all unit vectors
orthogonal to $v_{1,K}$. Applying IMS to these vectors yields the
lower bound with constant $\min\{\gamma_0,\gamma_1,\gamma_2\}$.
Local trial vectors, projected onto $v_{1,K}^{\perp}$, give the
matching upper bounds. This argument requires local compactness and
the first two stable local eigenvalues, without requiring convergence
of the higher global eigenvalues.

We also compare the local boundary model with the formal diffusion
approximation. The latter has drift $V$ and infinitesimal variance
$x(\widetilde\lambda(x)+\widetilde\mu(x))/K$, corresponding to
the coefficient structure in \cite[Eq.~(1.1)]{Yan-Zhou-continuous}.
The Taylor expansion for fixed smooth test functions does not control
profiles on the microscopic boundary scale $n=O(1)$.
Subsection~\ref{subsec:diffusion-comparison} gives an explicit example
of this limitation and shows that the frozen discrete and diffusion
operators nevertheless have the same spectral bottom. A comparison
of the global gap limits would additionally require estimates near
both interior equilibria and in the tail.

Section~\ref{sec:self-adjoint} develops the Jacobi formulation and
the IMS partition. The local approximation and spectral estimates in
Section~\ref{sec:local-estimates} are then combined in
Section~\ref{sec:gap-proof} to prove
\resultref{Theorem~}{thm:spectral-gap-limit}.
Section~\ref{sec:consequences} gives the consequences for the
$Q$-process and the examples, followed by the diffusion comparison
and questions on higher eigenvalues.
Appendix~\ref{sec:frozen-boundary-spectrum} records the frozen
discrete spectrum, and Appendix~\ref{sec:frozen-diffusion-bottom}
proves the corresponding diffusion spectral-bottom identity.

Table~\ref{tab:notation} collects the notation used throughout the paper.
The variational arguments are carried out over real Hilbert spaces.
For a positive weight $w=(w_n)_{n\ge1}$, we write
\[
\ell^2(w):=\left\{u:\sum_{n\ge1}w_nu_n^2<\infty\right\},
\qquad \langle u,v\rangle_w:=\sum_{n\ge1}w_nu_nv_n.
\]
Unsubscripted inner products and $\|\cdot\|_2$ refer to
$\ell^2(\mathbb N^*)$ with counting measure. The notation $L^2$
without a specified measure means $L^2(\mathbb R,dy)$; its inner
product is written $\langle\cdot,\cdot\rangle_{L^2}$.
We use $Q(z,w)$ for the symmetric bilinear form associated with a
quadratic form $Q(z)=Q(z,z)$, and $D(Q)$ for its form domain.
For a probability measure $m$, $m(g)$ denotes its integral of $g$.
When $K$ is clear from the context, we write $\lambda_n,\mu_n,\pi_n$
for $\lambda_n^{(K)},\mu_n^{(K)},\pi_n^{(K)}$. The Jacobi
coefficients $a_n,W_n$ introduced below also depend on $K$.
We use $O(\cdot)$ for unspecified upper bounds and remainder estimates.
Unless a different limit is stated, $O(\cdot)$ and $o(\cdot)$ refer
to $K\to\infty$. Implied constants may depend on the fixed rate
functions and cutoff profiles; dependence on additional parameters
is indicated by subscripts or stated locally. For a positive
comparison scale $G_K$, the notation $F_K=O_\delta(G_K)$ means that
$|F_K|/G_K$ remains bounded as $K\to\infty$ for each fixed $\delta$,
whereas $F_K=o_\delta(G_K)$ means that this ratio tends to zero.
Uniformity in lattice indices, vectors, and other variables is
specified in the relevant estimates. Constants retained explicitly
in inequalities are introduced with their parameter dependence
where they are used. A constant chosen for a definition or a
subsequent parameter choice retains that value throughout the
argument. We write $F_K\sim G_K$ when $F_K/G_K\to1$.

In the localization indices, $b$ denotes the boundary, $0$ the
intermediate region, and $r$ the right tail; thus $\chi_{0,K}$ is
distinct from the boundary cutoff $\chi_{b,K}$ associated with
$\gamma_0$. The superscript $(0)$ on frozen-model quantities refers
to evaluation of the per-capita rates at the boundary, not to $K=0$.

\begin{table}[!tbp]
\centering
\caption{Principal notation. The local objects are defined in Sections~\ref{sec:self-adjoint}--\ref{sec:local-estimates}.}
\label{tab:notation}
\small
\renewcommand{\arraystretch}{1.17}
\begin{tabular}{@{}p{0.29\textwidth}p{0.67\textwidth}@{}}
\hline
Symbol & Meaning \\
\hline
$K$; $n$; $x=n/K$ & Population scale, population size, and density coordinate. \\
$X^K$; $\tau_0$ & Birth--death process and its extinction time. \\
$\widetilde\lambda,\widetilde\mu$ & Per-capita birth and death rates; the total rates are $\lambda_n^{(K)}=n\widetilde\lambda(n/K)$ and $\mu_n^{(K)}=n\widetilde\mu(n/K)$. \\
$x_1,x_2$ & Unstable threshold and positive stable equilibrium. \\
$V,H$ & Deterministic drift and logarithmic potential, with $H(x_2)=0$. \\
$\pi^{(K)}$ & Unnormalized reversible weights of the killed chain. \\
$L_K$; $\rho_j(K)$ & Killed generator and the eigenvalues of $-L_K$, indexed from $j=1$. \\
$\gamma_0,\gamma_1,\gamma_2$; $\gamma_*$ & Boundary, threshold, and stable-equilibrium rates in \eqref{eq:local-rates}; $\gamma_*=\min_i\gamma_i$. \\
$\varphi_{1,K}$; $Z_K$ & Positive principal eigenfunction in $\ell^2(\pi^{(K)})$ and its squared norm. \\
$S_K$; $\mathcal H_K$; $Q_K$ & Jacobi transform, transformed operator, and its closed quadratic form. \\
$a_n$; $W_n$ & Jacobi edge coefficient and potential; their dependence on $K$ is suppressed. \\
$v_{1,K}$ & Normalized positive ground state $Z_K^{-1/2}S_K\varphi_{1,K}$ of $\mathcal H_K$. \\
$\nu^{(K)}$; $m^{(K)}$ & Quasi-stationary law and invariant law of the $Q$-process; their weights are proportional to $\pi_n^{(K)}\varphi_{1,K}(n)$ and $\pi_n^{(K)}\varphi_{1,K}(n)^2$, respectively. \\
$\mathcal L_K^Q$; $P_t^{Q,K}$; $C_{\mathrm P}(K)$ & $Q$-process generator, semigroup, and optimal Poincar\'e constant. \\
$\ell_K$; $d_K$; $n_i(K)$ & Cutoff width $K^{3/4}$, rescaled mesh $K^{-1/2}$, and lattice center $\lfloor Kx_i\rfloor$. \\
$I_{\sigma,K}$; $J_{\sigma,K}$; $\chi_{\sigma,K}$ & Core regions, enlarged regions, and quadratic partition of unity, with $\sigma\in\{b,0,1,2,r\}$. \\
$E_{i,K}$; $Q_{i,K}$; $\lambda_{j,K}^{(i)}$ & Local subspaces, restricted forms, and local eigenvalues, indexed from $j=0$; $i=1,2$. \\
$D_i$; $b_i$; $q_i$ & Coefficient $x_i\widetilde\lambda(x_i)$, drift derivative $V'(x_i)$, and limiting oscillator form. \\
$\alpha_0,\beta_0$; $\mathcal H_b^{(0)}$ & Rates $\widetilde\lambda(0),\widetilde\mu(0)$ and the frozen boundary Jacobi operator. \\
\hline
\end{tabular}
\end{table}

\section{Jacobi formulation and IMS localization}
\label{sec:self-adjoint}
The proof of \resultref{Theorem~}{thm:spectral-gap-limit} is based on
three variational bounds for the Jacobi operator associated with the
killed generator. We first construct this operator and then derive
the discrete IMS identity used to estimate its quadratic form.

\subsection{Self-adjoint formulation and variational reduction}
\label{subsec:self-adjoint-realization}
Throughout this subsection, $K$ is fixed.

To prove \eqref{eq:reversible-weight-decay}, set
\[
B:=\sup_{x\ge0}\frac{\widetilde\mu'(x)}{\widetilde\mu(x)}<\infty,
\qquad d_n:=\mu_n^2\pi_n.
\]
The logarithmic derivative bound in {\rm(A3)} gives
\[
\frac{\mu_{n+1}}{\mu_n}
=\left(1+\frac1n\right)
 \exp\left(\int_{n/K}^{(n+1)/K}
       \frac{\widetilde\mu'(s)}{\widetilde\mu(s)}\,ds\right)
\le\left(1+\frac1n\right)e^{B/K}.
\]
By detailed balance and the first limit in {\rm(A3)},
\[
\frac{d_{n+1}}{d_n}
=\frac{\mu_{n+1}}{\mu_n}\frac{\lambda_n}{\mu_n}
\le\left(1+\frac1n\right)e^{B/K}
 \frac{\widetilde\lambda(n/K)}{\widetilde\mu(n/K)}
\longrightarrow0
\qquad(n\to\infty).
\]
Thus $d_{n+1}\le d_n/2$ for all sufficiently large $n$, and
$d_n\to0$, proving \eqref{eq:reversible-weight-decay}.

For fixed $K$, assumption {\rm(A3)} gives $\mu_n/\lambda_n\to\infty$.
Consequently,
\[
R_n:=\prod_{j=1}^n\frac{\mu_j}{\lambda_j}\longrightarrow\infty,
\qquad \sum_{n\ge1}R_n=\infty.
\]
The series in the non-explosion criterion of
\cite[Theorem~10]{Meleard-Villemonais2012} dominates
$\sum_nR_n$, so the chain is non-explosive. The same divergence is
the absorption criterion in \cite[Proposition~12]{Meleard-Villemonais2012}.
Hence $\mathbb P_n(\tau_0<\infty)=1$ for every $n\ge1$.

On finitely supported sequences define the Dirichlet form
\[
\mathcal E_K(u,u)=\mu_1\pi_1|u_1|^2+
\sum_{n\ge1}\lambda_n\pi_n|u_{n+1}-u_n|^2.
\]
This densely defined nonnegative form is closable on $\ell^2(\pi)$.
Indeed, convergence in $\ell^2(\pi)$ implies coordinatewise convergence.
For an $\mathcal E_K$-Cauchy sequence, the weighted edge differences
converge in $\ell^2$, and the boundary value
at $n=1$ converges. If its $\ell^2(\pi)$ limit is zero, all these
limits vanish, which proves closability. We use $\mathcal E_K$ also
for its closure. Convergence of the weighted edge differences extends
the series representation to the full form domain.

The representation theorem \cite[Theorem~2.14]{Teschl2014} associates a nonnegative self-adjoint operator with this closed form. We identify its semigroup with that of the process killed at $0$.
Restricting the form to sequences supported on $\{1,\ldots,N\}$
gives the generator killed upon hitting either $0$ or $N+1$.
Non-explosion guarantees that these finite-dimensional transition kernels
increase pointwise to those of the infinite-state killed chain.
For $f\in\ell^2(\pi)$ and $\alpha>0$, the resolvent associated with
$\mathcal E_K$ is the unique $u\in D(\mathcal E_K)$ satisfying
\[
\mathcal E_K(u,v)+\alpha\langle u,v\rangle_\pi
=\langle f,v\rangle_\pi,\qquad v\in D(\mathcal E_K).
\]
The finite-dimensional variational solutions converge in form norm,
since finitely supported sequences form a core. Taking Laplace
transforms of the increasing transition kernels identifies the limiting
resolvent with that of the killed process. We denote this nonnegative
self-adjoint operator by $-L_K$.

The unitary transformation and the associated Jacobi operator are
\[
S_K:\ell^2(\pi^{(K)})\longrightarrow\ell^2(\mathbb N^*),
\qquad (S_Ku)_n=\sqrt{\pi_n}\,u_n,
\qquad\mathcal H_K=S_K(-L_K)S_K^{-1}.
\]
To apply the variational principle, we express the transformed quadratic form in terms of nearest-neighbor differences. Let $c_{00}(\mathbb N^*)$ denote the space of finitely supported sequences. The closed quadratic form of $\mathcal H_K$ is $Q_K(z):=\mathcal E_K(S_K^{-1}z,S_K^{-1}z)$. On $c_{00}(\mathbb N^*)$, detailed balance gives
\[
Q_K(z)=\sum_{n\ge1}a_n(z_{n+1}-z_n)^2+\sum_{n\ge1}W_nz_n^2,
\]
where
\[
a_n=\sqrt{\lambda_n\mu_{n+1}},\qquad
W_n=\lambda_n+\mu_n-a_n-a_{n-1},\qquad a_0=0.
\]
The form domain $D(Q_K)$ is the completion of $c_{00}$ under the norm $(Q_K(z)+\|z\|_2^2)^{1/2}$. The estimates below justify extending this series representation to the full domain and show that multiplication by our later localization cutoffs preserves membership in $D(Q_K)$. These form-domain estimates are for fixed $K$; no uniformity in $K$ is required here.

By the definition of $B$, assumption {\rm(A3)}, and monotonicity,
\[
\frac{\mu_{n+1}}{\mu_n}\le(1+1/n)e^{B/K},\qquad
\frac{a_n}{\mu_n}
=\left(\frac{\lambda_n}{\mu_n}\frac{\mu_{n+1}}{\mu_n}\right)^{1/2}
\longrightarrow0.
\]
Similarly,
\[
\frac{a_{n-1}}{\mu_n}
=\left(\frac{\lambda_{n-1}}{\mu_n}\right)^{1/2}
\le\left(\frac{\lambda_{n-1}}{\mu_{n-1}}\right)^{1/2}
\longrightarrow0.
\]
Hence $W_n/\mu_n\to1$ as $n\to\infty$, so $W_n$ is negative at only finitely many indices. Fix $c_K\ge1$ such that $W_n+c_K\ge1$ for all $n$, and define the shifted form on $c_{00}$ by
\[
Q_K^+(z):=Q_K(z)+c_K\|z\|_2^2
=\sum_{n\ge1}a_n|z_{n+1}-z_n|^2
 +\sum_{n\ge1}(W_n+c_K)|z_n|^2.
\]
Since $Q_K\ge0$, $(Q_K^+)^{1/2}$ defines an equivalent form norm. If a sequence from $c_{00}$ is Cauchy in this norm, its edge differences and coordinate values are Cauchy in the respective weighted $\ell^2$ spaces. Coordinatewise convergence identifies their limits with the differences and values of the $\ell^2$ limit. Thus the series representation of $Q_K^+$, and hence of $Q_K$, extends to every element of $D(Q_K)$.

These bounds also imply compactness of the resolvent. Since $\mu_n\ge n\widetilde\mu(0)$ and $W_n/\mu_n\to1$, there exists an integer $N_K$ such that
\[
W_n\ge\frac{\mu_n}{2}\ge\frac{n\widetilde\mu(0)}2,
\qquad n>N_K.
\]
For $N\ge N_K$ and $z\in D(Q_K)$,
\[
\sum_{n>N}|z_n|^2
\le\frac{2}{\widetilde\mu(0)(N+1)}
\bigl(Q_K(z)+c_K\|z\|_2^2\bigr).
\]
Thus any set bounded in form norm has uniformly small $\ell^2$ tails. Projections onto the first $N$ coordinates give relatively compact sets, so $D(Q_K)$ embeds compactly into $\ell^2$. It follows that $\mathcal H_K$, and equivalently $-L_K$, possess compact resolvents.

The operator has trivial kernel: zero energy would force the boundary value and all successive differences in the Dirichlet form to vanish, hence $\rho_1(K)>0$. From the inequality $\mathcal E_K(|u|,|u|)\le\mathcal E_K(u,u)$ we may take the principal eigenfunction nonnegative; positivity of transition rates in the recurrence then forces strict positivity. Given the Jacobi recurrence, an eigenvector is fully determined by its first entry, so a vanishing first component implies the whole vector vanishes. All eigenvalues are therefore simple, recovering the ordering in \eqref{eq:killed-spectrum}.

Define the normalized positive principal eigenvector
\[
v_{1,K}:=Z_K^{-1/2}S_K\varphi_{1,K}.
\]
It satisfies
\[
\mathcal H_Kv_{1,K}=\rho_1(K)v_{1,K},\qquad
\|v_{1,K}\|_2=1,\qquad m_n^{(K)}=v_{1,K}(n)^2.
\]

We will also need stability of the form domain under multiplication by cutoffs. Let $\chi$ be a bounded sequence whose differences are nonzero only on a finite set. For finitely supported $z$, we use the identity
\[
\chi_{n+1}z_{n+1}-\chi_nz_n
=\chi_{n+1}(z_{n+1}-z_n)+(\chi_{n+1}-\chi_n)z_n.
\]
It follows that
\[
\begin{aligned}
Q_K^+(\chi z)
&\le2\|\chi\|_\infty^2Q_K^+(z)
+2\sum_n a_n|\chi_{n+1}-\chi_n|^2|z_n|^2\\
&\le2\left(\|\chi\|_\infty^2
+\sup_{n\ge1}a_n|\chi_{n+1}-\chi_n|^2\right)Q_K^+(z).
\end{aligned}
\]
The coefficient in the last line is finite because $\chi$ changes across only finitely many edges; the second inequality uses $Q_K^+(z)\ge\|z\|_2^2$. Approximation by finitely supported vectors in form norm then gives $\chi z\in D(Q_K)$ for every $z\in D(Q_K)$. In particular, the form domain is preserved by the localization cutoffs used below and by indicators of half-lines starting at a fixed integer. The IMS error estimate in Subsection~\ref{subsec:localization-partition} will be uniform in $K$.

By the Courant--Fischer variational principle in its quadratic-form version
\cite[Theorem~4.12]{Teschl2014},
\[
\rho_2(K)=\inf_{\substack{z\in D(Q_K),\ \|z\|_2=1\\z\perp v_{1,K}}}Q_K(z).
\]
A local Gaussian trial state from \resultref{Lemma~}{lem:x2-window-kernel-gap}(i) shows $\rho_1(K)\to0$. It therefore remains to prove $\rho_2(K)\to\gamma_*$. The following three propositions give the required lower and upper bounds.

\begin{proposition}\label{prop:spectral-gap-liminf-lower}
For every $0<\varepsilon<\gamma_*$, there exists $K_0$ such that
\[
Q_K(z)\ge(\gamma_*-\varepsilon)\|z\|_2^2,
\qquad z\in D(Q_K),\quad z\perp v_{1,K},\quad K\ge K_0.
\]
Consequently,
$\liminf_{K\to\infty}(\rho_2(K)-\rho_1(K))\ge\gamma_*$.
\end{proposition}

\begin{proposition}\label{prop:left-boundary-limsup-gap}
The boundary rate gives the upper bound
\[
\limsup_{K\to\infty}(\rho_2(K)-\rho_1(K))\le\gamma_0.
\]
\end{proposition}

\begin{proposition}\label{prop:spectral-gap-upper-bound-interior}
The two interior rates give the upper bound
\[
\limsup_{K\to\infty}(\rho_2(K)-\rho_1(K))
\le\min\{\gamma_1,\gamma_2\}.
\]
\end{proposition}

These propositions are proved in Subsections~\ref{sec:lower-bound} and \ref{sec:upper-bound}.

\subsection{The discrete IMS formula}
\label{subsec:localization-partition}
To compare the contributions from the boundary and the two equilibria,
we use the discrete IMS identity for the nearest-neighbor Jacobi
operator; see also \cite[Lemma~3.3(a), Eq.~(3.15)]{Klein-Rosenberger2009}.
We construct the cutoffs first and then estimate the localization error.

Fix a small $\delta>0$ such that
\[
0<\delta<x_1<x_2<\frac1\delta.
\]
The parameter $\delta$ may be reduced later when applying the boundary and tail estimates. Set
\[
\ell_K:=K^{3/4}.
\]
This scale satisfies
\[
\ell_K\to\infty,
\qquad
\frac{\ell_K}{K}\to0,
\qquad
\frac{K}{\ell_K^2}=K^{-1/2}\to0.
\]
The localization windows grow on the microscopic scale but shrink relative to population size $K$, yielding an $o(1)$ localization error.

Define
\[
n_i(K):=\lfloor Kx_i\rfloor,
\qquad i=1,2.
\]
We index the five localization regions by $\Sigma:=\{b,0,1,2,r\}$.
Here $b$ denotes the left boundary layer and $r$ the far-right tail. The indices $1,2$ correspond to neighborhoods of the two equilibria, while $0$ denotes the intermediate regions separating them from each other and from the boundary and tail.

The core regions are
\[
I_{b,K}:=\{1,\dots,\lfloor\delta K\rfloor\},
\]
\[
I_{1,K}:=
\{n\in\mathbb N^*: |n-n_1(K)|\le \ell_K\},
\qquad
I_{2,K}:=
\{n\in\mathbb N^*: |n-n_2(K)|\le \ell_K\},
\]
\[
I_{0,K}:=
\left\{
n\in\mathbb N^*:
\lfloor\delta K\rfloor+\ell_K
\le n\le
\lfloor K/\delta\rfloor-\ell_K,\ 
|n-n_1(K)|\ge 2\ell_K,\ 
|n-n_2(K)|\ge 2\ell_K
\right\},
\]
\[
I_{r,K}:=\{n\in\mathbb N^*: n>\lfloor K/\delta\rfloor\}.
\]

Their enlargements are
\[
J_{b,K}:=
\{n\in\mathbb N^*: n\le \lfloor\delta K\rfloor+\ell_K\},
\]
\[
J_{1,K}:=
\{n\in\mathbb N^*: |n-n_1(K)|\le 2\ell_K\},
\qquad
J_{2,K}:=
\{n\in\mathbb N^*: |n-n_2(K)|\le 2\ell_K\},
\]
\[
J_{r,K}:=
\{n\in\mathbb N^*: n\ge \lfloor K/\delta\rfloor-\ell_K\},
\]
\[
J_{0,K}:=
\Bigl\{
n\in\mathbb N^*:
\lfloor\delta K\rfloor\le n\le \lfloor K/\delta\rfloor,\ 
|n-n_1(K)|\ge \ell_K,\ 
|n-n_2(K)|\ge \ell_K
\Bigr\}.
\]

For sufficiently large $K$, we have $I_{\sigma,K}\subset J_{\sigma,K}$ for every $\sigma\in\Sigma$, and the enlarged sets cover all of $\mathbb N^*$. Since $\ell_K=o(K)$, the core regions are pairwise disjoint for large $K$, and the enlarged sets overlap only in transition layers of width $O(\ell_K)$.

To construct a quadratic partition of unity subordinate to these regions, let
\[
\vartheta(t):=
\begin{cases}
1, & 0\le t\le \dfrac13,\\[0.6em]
2-3t, & \dfrac13<t<\dfrac23,\\[0.6em]
0, & t\ge \dfrac23.
\end{cases}
\]
This function satisfies $0\le\vartheta\le1$ and
\[
|\vartheta(t)-\vartheta(s)|\le3|t-s|,
\qquad s,t\ge0.
\]
For a subset $A\subset\mathbb N^*$, define the discrete distance by
\[
d(n,A):=\inf_{m\in A}|n-m|.
\]
Preliminary cutoffs are given by
\[
\widetilde\chi_{\sigma,K}(n)
:=
\vartheta\left(\frac{d(n,I_{\sigma,K})}{\ell_K}\right),
\qquad
\sigma\in\Sigma.
\]
They satisfy
\[
\widetilde\chi_{\sigma,K}\equiv1
\quad\text{on }I_{\sigma,K},
\qquad
\operatorname{supp}\widetilde\chi_{\sigma,K}\subset J_{\sigma,K},
\]
\[
|\widetilde\chi_{\sigma,K}(n+1)
-
\widetilde\chi_{\sigma,K}(n)|
\le 3\ell_K^{-1}.
\]

The gaps between adjacent core components have width at most
the cutoff width, up to integer-rounding errors. Thus for every
lattice point $n\in\mathbb N^*$ we have
\[
\min_{\sigma\in\Sigma}d(n,I_{\sigma,K})
\le \frac{\ell_K}{2}+2.
\]
Any point outside all core regions lies in one of these gaps, of
width at most $\ell_K+O(1)$. For sufficiently large $K$,
at least one preliminary cutoff therefore
takes value at least $1/3$ for each $n$, so
\[
\sum_{\sigma\in\Sigma}
\widetilde\chi_{\sigma,K}(n)^2
\ge \frac19,
\qquad n\in\mathbb N^*.
\]

The normalized cutoffs are
\[
\chi_{\sigma,K}(n):=
\frac{\widetilde\chi_{\sigma,K}(n)}
{\left(
\sum_{\tau\in\Sigma}
\widetilde\chi_{\tau,K}(n)^2
\right)^{1/2}},
\qquad
\sigma\in\Sigma.
\]
By construction,
\[
\sum_{\sigma\in\Sigma}
\chi_{\sigma,K}(n)^2
=1,
\qquad n\in\mathbb N^*,
\]
\[
\chi_{\sigma,K}\equiv1
\quad\text{on }I_{\sigma,K},
\qquad
\operatorname{supp}\chi_{\sigma,K}\subset J_{\sigma,K},
\]
\[
|\chi_{\sigma,K}(n+1)-\chi_{\sigma,K}(n)|
=O(\ell_K^{-1}),
\]
uniformly in $n\in\mathbb N^*$ and $\sigma\in\Sigma$. The implied
constant depends only on the fixed profile $\vartheta$ and the
number of regions.

Figure~\ref{fig:ims-localization} shows the resulting partition. The intermediate region $I_{0,K}$ consists of three disconnected lattice components separating the boundary layer, the two equilibrium windows, and the far-right tail. Panel (b) shows how the squared cutoffs sum to one across each transition layer.

\begin{figure}[!tbp]
\centering
\begingroup
\definecolor{imsBoundary}{RGB}{0,114,178}
\definecolor{imsMiddle}{RGB}{101,112,126}
\definecolor{imsThreshold}{RGB}{213,94,0}
\definecolor{imsStable}{RGB}{0,135,100}
\definecolor{imsTail}{RGB}{117,82,151}
\definecolor{imsInk}{RGB}{32,39,48}
%
\newcommand{\imsband}[4]{%
  \fill[#1!17] (#2,{#4-.29}) rectangle (#3,{#4+.29});
  \draw[#1!60,line width=.45pt] (#2,{#4-.29}) rectangle (#3,{#4+.29});
}
\newcommand{\imscore}[4]{%
  \fill[#1] (#2,{#4-.115}) rectangle (#3,{#4+.115});
}
\pgfmathdeclarefunction{imstheta}{1}{%
  \pgfmathparse{max(0,min(1,2-3*#1))}%
}
\pgfmathdeclarefunction{imsinner}{1}{%
  \pgfmathparse{pow(imstheta(max(abs(#1)-1,0)),2)/
    (pow(imstheta(max(abs(#1)-1,0)),2)+
     pow(imstheta(max(2-abs(#1),0)),2))}%
}
\resizebox{\textwidth}{!}{%
\begin{tikzpicture}[
  x=.66cm,y=.58cm,
  font=\small,text=imsInk,
  line cap=round,line join=round,>=stealth
]
\node[anchor=west,font=\bfseries] at (-4.8,15.1)
  {(a) Five localization regions on the population axis};
\node[anchor=east,font=\footnotesize,text=imsMiddle] at (20.8,14.45)
  {Not to scale; integer rounding suppressed};
\imsband{imsMiddle}{0}{1}{13.65}
\node[anchor=west] at (1.3,13.65) {enlarged region $J_{\sigma,K}$};
\imscore{imsMiddle}{10.4}{11.4}{13.65}
\node[anchor=west] at (11.7,13.65) {core region $I_{\sigma,K}$};
\foreach \a/\b in {2/3,5/6,8/9,11/12,14/15,17/18}{
  \fill[black!3] (\a,5.62) rectangle (\b,11.64);
}
\foreach \x in {2,7,13,18}{
  \draw[imsMiddle!45,densely dotted,line width=.55pt]
    (\x,5.62)--(\x,11.9);
}
\draw[->,line width=.75pt] (0,11.9)--(20.8,11.9)
  node[right] {$n$};
\foreach \x/\lab in {0/{1},2/{\delta K},7/{n_1(K)},13/{n_2(K)},18/{K/\delta}}{
  \draw[line width=.7pt] (\x,11.81)--(\x,11.99);
  \node[above=4pt] at (\x,11.9) {$\lab$};
}
\node[above=4pt] at (20.2,11.9) {$+\infty$};
\node[anchor=west] at (-4.8,10.65) {$b$: boundary};
\node[anchor=west] at (-4.8,9.5)  {$0$: intermediate};
\node[anchor=west] at (-4.8,8.35) {$1$: threshold};
\node[anchor=west] at (-4.8,7.2)  {$2$: stable};
\node[anchor=west] at (-4.8,6.05) {$r$: right tail};
\imsband{imsBoundary}{0}{3}{10.65}
\imscore{imsBoundary}{0}{2}{10.65}
\foreach \a/\b in {2/6,8/12,14/18}{
  \imsband{imsMiddle}{\a}{\b}{9.5}
}
\foreach \a/\b in {3/5,9/11,15/17}{
  \imscore{imsMiddle}{\a}{\b}{9.5}
}
\imsband{imsThreshold}{5}{9}{8.35}
\imscore{imsThreshold}{6}{8}{8.35}
\imsband{imsStable}{11}{15}{7.2}
\imscore{imsStable}{12}{14}{7.2}
\path[fill=imsTail!17,draw=imsTail!60,line width=.45pt]
  (17,5.76)--(20.08,5.76)--(20.75,6.05)--(20.08,6.34)--(17,6.34)--cycle;
\imscore{imsTail}{18}{20.08}{6.05}
\draw[imsTail,->,line width=1.15pt] (19.75,6.05)--(20.75,6.05);
\node[anchor=west,font=\footnotesize,text=imsMiddle] at (-4.8,5.17)
  {The intermediate region ($\sigma=0$) is the union of three lattice intervals.};
\node at (7.8,4.43) {
  $\displaystyle
  \chi_{\sigma,K}=1\ \text{on }I_{\sigma,K},\qquad
  \operatorname{supp}\chi_{\sigma,K}\subset J_{\sigma,K},\qquad
  \sum_{\sigma\in\{b,0,1,2,r\}}\chi_{\sigma,K}^{\,2}=1.$
};
\node[anchor=west,font=\bfseries] at (-4.8,3.33)
  {(b) Zoom near an interior equilibrium ($i=1,2$)};
\node[anchor=east,font=\footnotesize,text=imsMiddle] at (20.8,2.75)
  {$n_i=n_i(K)$};
\imsband{imsThreshold}{2}{18}{2.27}
\imscore{imsThreshold}{6}{14}{2.27}
\node[anchor=east,font=\footnotesize] at (1.6,2.27)
  {$I_{i,K}\subset J_{i,K}$};
\fill[imsMiddle!8] (2,-1.0) rectangle (6,1.55);
\fill[imsMiddle!8] (14,-1.0) rectangle (18,1.55);
\foreach \x in {2,6,14,18}{
  \draw[imsMiddle!40,densely dashed,line width=.45pt]
    (\x,-1.0)--(\x,1.55);
}
\draw[imsMiddle!35,densely dotted] (10,-1.0)--(10,1.55);
\draw[->,line width=.7pt] (0,-1.0)--(20.8,-1.0) node[right] {$n$};
\draw[->,line width=.7pt] (0,-1.0)--(0,1.65);
\node[anchor=east] at (-.2,-1.0) {$0$};
\node[anchor=east] at (-.2,1.2) {$1$};
\node[anchor=east,font=\footnotesize] at (-.2,1.65) {squared cutoffs};
\draw[imsMiddle!18,line width=.4pt] (0,1.2)--(20,1.2);
\draw[imsMiddle,densely dashed,line width=1.2pt,
      domain=-2.5:2.5,samples=301]
  plot ({10+4*\x},{-1+2.2*(1-imsinner(\x))});
\draw[imsThreshold,line width=1.35pt,
      domain=-2.5:2.5,samples=301]
  plot ({10+4*\x},{-1+2.2*imsinner(\x)});
\node[text=imsThreshold,fill=white,inner sep=2pt] at (10,.54)
  {$\chi_{i,K}^{\,2}$};
\node[text=imsMiddle,fill=white,inner sep=2pt] at (1.15,.76)
  {$\chi_{0,K}^{\,2}$};
\node[text=imsMiddle,fill=white,inner sep=2pt] at (19.0,.76)
  {$\chi_{0,K}^{\,2}$};
\node[font=\footnotesize,fill=white,inner sep=2pt] at (10,-.48)
  {$\chi_{i,K}^{\,2}+\chi_{0,K}^{\,2}=1$};
\foreach \x/\lab in {
  2/{n_i-2\ell_K},6/{n_i-\ell_K},10/{n_i},
  14/{n_i+\ell_K},18/{n_i+2\ell_K}}{
  \draw (\x,-1.0)--(\x,-1.12);
  \node[below=4pt,font=\footnotesize] at (\x,-1.0) {$\lab$};
}
\draw[<->,imsMiddle,line width=.6pt] (2,-1.99)--(6,-1.99);
\draw[<->,imsMiddle,line width=.6pt] (14,-1.99)--(18,-1.99);
\node[below=2pt,font=\footnotesize] at (4,-1.99) {overlap: $\ell_K$};
\node[below=2pt,font=\footnotesize] at (16,-1.99) {overlap: $\ell_K$};
\draw[<->,imsThreshold,line width=.6pt] (6,-1.99)--(14,-1.99);
\node[below=2pt,font=\footnotesize] at (10,-1.99)
  {core width: $2\ell_K$};
\draw[imsMiddle!40,line width=.5pt] (-4.8,-2.88)--(20.8,-2.88);
\node at (8,-3.8) {
  $\displaystyle
  \underbrace{\sqrt K}_{\text{fluctuations}}
  =o(\ell_K),\quad
  \underbrace{\ell_K}_{\text{cutoff width}}=o(K),\qquad
  |\chi_{\sigma,K}(n+1)-\chi_{\sigma,K}(n)|
  =O(\ell_K^{-1}).
  $
};
\node at (8,-5.5) {
  $\displaystyle
  Q_K(\zeta)\ge
  \sum_{\sigma\in\{b,0,1,2,r\}}Q_K(\chi_{\sigma,K}\zeta)
  -\underbrace{C_{\mathrm{IMS},\delta}K^{-1/2}\|\zeta\|_2^2}_{\text{IMS localization error}}.
  $
};
\end{tikzpicture}%
}
\endgroup
\caption{The regions and cutoffs used in the discrete IMS localization.
(a) Dark bands denote the main regions $I_{\sigma,K}$ and pale bands
their enlargements $J_{\sigma,K}$.
(b) Squared normalized cutoffs near an interior equilibrium
$n_i(K)$, $i=1,2$; only $\chi_{i,K}$ and $\chi_{0,K}$ can be nonzero
there. The shaded overlaps have width of order $\ell_K=K^{3/4}$.
The schematic is not to scale and suppresses integer rounding.
The error constant is the one fixed in \eqref{eq:discrete-IMS}.}
\label{fig:ims-localization}
\end{figure}

Define
\[
    \Gamma_K(n)
    :=
    \sum_{\sigma\in\Sigma}
    \bigl(
        \chi_{\sigma,K}(n+1)-\chi_{\sigma,K}(n)
    \bigr)^2,
    \qquad n\geq1.
\]
The partition-of-unity identity $\sum_{\sigma\in\Sigma}\chi_{\sigma,K}(n)^2=1$ gives
\[
\begin{aligned}
&\sum_{\sigma\in\Sigma}
\bigl(
    \chi_{\sigma,K}(n+1)\zeta_{n+1}
    -
    \chi_{\sigma,K}(n)\zeta_n
\bigr)^2
\\
&\qquad
=
(\zeta_{n+1}-\zeta_n)^2
+
\zeta_n\zeta_{n+1}\Gamma_K(n).
\end{aligned}
\]
Summing over $n$ and applying the same partition identity to the potential terms of $Q_K$ gives the discrete IMS formula
\begin{equation}
\label{eq:discrete-IMS-identity}
    \sum_{\sigma\in\Sigma}
    Q_K(\chi_{\sigma,K}\zeta)
    =
    Q_K(\zeta)
    +
    \sum_{n\geq1}
    a_n\zeta_n\zeta_{n+1}\Gamma_K(n).
\end{equation}
The equality holds for finitely supported $\zeta\in c_{00}(\mathbb N^*)$ and extends by density to the full form domain $D(Q_K)$. Rearranging gives the lower bound
\[
    Q_K(\zeta)
    \geq
    \sum_{\sigma\in\Sigma}
    Q_K(\chi_{\sigma,K}\zeta)
    -
    \sum_{n\geq1}
    a_n|\zeta_n\zeta_{n+1}|\Gamma_K(n).
\]

By construction,
\[
    \Gamma_K(n)=O(\ell_K^{-2})
\]
uniformly in $n$. The implied constant depends only on the fixed
cutoff profile and the number of regions. Moreover, $\Gamma_K(n)$
is nonzero only in transition layers, all contained in
$1\le n\le2K/\delta$ for sufficiently large $K$. On these layers,
$n/K$ and $(n+1)/K$ lie in a fixed compact interval depending on
$\delta$. The per-capita rates are uniformly bounded there, giving
\[
    a_n
    =
    \sqrt{\lambda_n\mu_{n+1}}
    =O_\delta(K)
\]
uniformly on $\operatorname{supp}\Gamma_K$. Applying the elementary
estimate $\sum_{n\ge1}|\zeta_n\zeta_{n+1}|\le\|\zeta\|_2^2$ gives
\[
    \sum_{n\geq1}
    a_n|\zeta_n\zeta_{n+1}|\Gamma_K(n)
    =O_\delta\left(\frac{K}{\ell_K^2}\right)\|\zeta\|_2^2,
\]
uniformly in $\zeta\in D(Q_K)$. Since $K/\ell_K^2=K^{-1/2}$,
we can fix a constant
$C_{\mathrm{IMS},\delta}>0$, independent of $K$ and $\zeta$, such that
\begin{equation}
\label{eq:discrete-IMS}
    Q_K(\zeta)
    \geq
    \sum_{\sigma\in\Sigma}
    Q_K(\chi_{\sigma,K}\zeta)
    -
    C_{\mathrm{IMS},\delta}K^{-1/2}
    \|\zeta\|_2^2.
\end{equation}
For unit vectors, the localization error is $O_\delta(K^{-1/2})$,
as needed for the variational lower bound.
\section{Local spectral estimates}
\label{sec:local-estimates}

We now estimate the quadratic form on each region of the partition constructed in Subsection~\ref{subsec:localization-partition}. Throughout this section, $\delta$ is fixed and sufficiently small. Unless otherwise stated, constants may depend on the rate functions and on $\delta$, but not on $K$.

\subsection{Interior approximation and compactness}
\label{subsec:interior-approximation}

Both interior windows can be treated using the same sampling and
interpolation estimates. Sampling approximates the norm and energy of
smooth test functions, while interpolation gives compactness in
$L^2(\mathbb R)$ for unit vectors with uniformly bounded energy and
preserves the norm in the limit. The following lemma also establishes
lower semicontinuity of the energy and convergence of inner products.
These estimates will be used in
\resultref{Lemmas~}{lem:x1-window-precise}
and~\resultref{}{lem:x2-window-kernel-gap}.

For $i=1,2$, define
\[
E_{i,K}:=\{z\in\ell^2(\mathbb N^*):\operatorname{supp}z\subset J_{i,K}\},
\qquad Q_{i,K}:=Q_K|_{E_{i,K}},
\]
and set
\[
D_i:=x_i\widetilde\lambda(x_i),\qquad b_i:=V'(x_i),\qquad
d_K:=K^{-1/2},\qquad y_{i,n}:=\frac{n-n_i(K)}{\sqrt K}.
\]
Thus $b_1=\gamma_1$ and $b_2=-\gamma_2$. In the notation $\psi_{i,j}$ used below, $i$ identifies the equilibrium window, $j$ is the excitation index, and $j=0$ corresponds to the local ground state.

Since $E_{i,K}$ is finite-dimensional and contained in $c_{00}$, the restricted quadratic form corresponds to a finite symmetric matrix. Vectors in this subspace are extended by zero outside $J_{i,K}$, so their energy includes the contributions from edges crossing the window boundary.

\begin{lemma}[Interior approximation]
\label{lem:interior-approximation}
For each $i\in\{1,2\}$, let $q_i$ denote the oscillator-type limiting quadratic form given in \eqref{eq:interior-limit-form}. The sampling operator $\mathcal S_{i,K}$ and interpolation operator $\mathcal I_{i,K}$ defined below satisfy the following properties.

For every $\psi\in C_c^\infty(\mathbb R)$,
\[
\|\mathcal S_{i,K}\psi\|_2^2\longrightarrow\|\psi\|_{L^2}^2,
\qquad Q_K(\mathcal S_{i,K}\psi)\longrightarrow q_i(\psi).
\]

For any sequence of unit vectors $z^{(K)}\in E_{i,K}$ with uniformly bounded quadratic-form energy, a subsequence of the interpolations converges strongly in $L^2(\mathbb R)$ and weakly in $H^1(\mathbb R)$ to a unit-norm limit $\Psi\in D(q_i)$. Along this subsequence,
\[
q_i(\Psi)\le\liminf_{K\to\infty}Q_K(z^{(K)}).
\]

For two sequences in the same local subspaces with uniformly bounded norms and energies, strong convergence of their interpolations implies convergence of their discrete inner products to the inner product of the limits. In particular, orthogonality is preserved in the limit.
\end{lemma}

\begin{proof}
Introduce
\[
f(x)=x\widetilde\lambda(x),\qquad
g(x)=x\widetilde\mu(x),\qquad
q=\sqrt f,\qquad r=\sqrt g.
\]
On compact subsets of $(0,\infty)$, these functions are $C^2$. With $x=n/K$ and mesh size $h=1/K$, we have
\[
a_n=Kq(x)r(x+h),\qquad a_{n-1}=Kq(x-h)r(x),
\]
and hence
\[
W_n=K(q(x)-r(x))^2+
Kq(x)(r(x)-r(x+h))+Kr(x)(q(x)-q(x-h)).
\]

Taylor expansion on a slightly enlarged compact interval, where the second derivatives are uniformly bounded, gives
\begin{equation}\label{eq:local-potential-expansion}
W_n=K(q-r)^2(x)+(rq'-qr')(x)+O(K^{-1}).
\end{equation}

At each equilibrium point,
\[
q(x_i)=r(x_i)=\sqrt{D_i},\qquad
(q-r)'(x_i)=\frac{b_i}{2\sqrt{D_i}},\qquad
(rq'-qr')(x_i)=\frac{b_i}{2}.
\]

By transversality, $|q(x)-r(x)|/|x-x_i|$ is bounded away from zero
in a sufficiently small fixed neighborhood of $x_i$, with $x\ne x_i$.
For large $K$, all lattice points in $J_{i,K}$ and their neighbors
lie in this neighborhood. Set
\[
x-x_i=d_Ky_{i,n}+\epsilon_{i,K},\qquad
\epsilon_{i,K}:=n_i(K)/K-x_i\in[-K^{-1},0].
\]
Then $K|x-x_i|^2\ge y_{i,n}^2/2-K^{-1}$. Positivity of $q,r$ and
the error bound in \eqref{eq:local-potential-expansion} give constants
$c,C>0$, independent of $K$ and $i\in\{1,2\}$, such that
\begin{equation}\label{eq:local-coercive-bounds}
cK\le a_n\le CK,\qquad
W_n\ge c y_{i,n}^2-C\quad(n\in J_{i,K}),
\end{equation}
where the bound on $a_n$ applies to every edge adjacent to $J_{i,K}$.
Since $|n/K-x_i|=O(K^{-1/4})$ uniformly on $J_{i,K}$, the entire
window stays inside the fixed neighborhood. Uniform $o(1)$ Taylor
remainders are only required over fixed intervals in the rescaled
coordinate $y_{i,n}$.

For fixed $R>0$, Taylor expansion on $|y_{i,n}|\le R$ gives
\[
q(x)-r(x)=\frac{b_i}{2\sqrt{D_i}}d_Ky_{i,n}+O_R(K^{-1}),
\qquad (rq'-qr')(x)=\frac{b_i}{2}+O_R(K^{-1/2}).
\]
Substitution into the coefficient formulas yields
\begin{equation}\label{eq:local-fixed-window-expansions}
a_n=D_iK+O_R(\sqrt K),\qquad
W_n=\frac{b_i^2}{4D_i}y_{i,n}^2+\frac{b_i}{2}+O_R(K^{-1/2}).
\end{equation}

The limiting quadratic form is therefore
\begin{equation}\label{eq:interior-limit-form}
q_i(\psi)=D_i\int_{\mathbb R}|\psi'|^2\,dy+
\int_{\mathbb R}p_i(y)|\psi|^2\,dy,
\qquad p_i(y):=\frac{b_i^2}{4D_i}y^2+\frac{b_i}{2},
\end{equation}
with form domain $H^1(\mathbb R)\cap L^2(\mathbb R,y^2dy)$.

Given $\psi\in C_c^\infty(\mathbb R)$, define the sampled sequence
\[
(\mathcal S_{i,K}\psi)_n=d_K^{1/2}\psi(y_{i,n}).
\]
For fixed $\psi$ and sufficiently large $K$, this vector is supported in $I_{i,K}$, whose width in rescaled coordinates is of order $K^{1/4}$. Its squared norm is a Riemann sum for the squared $L^2$ norm. The gradient contribution to the energy is
\[
\sum_n a_n\bigl((\mathcal S_{i,K}\psi)_{n+1}
                 -(\mathcal S_{i,K}\psi)_n\bigr)^2
=d_K\sum_n\frac{a_n}{K}
\left(\frac{\psi(y_{i,n}+d_K)-\psi(y_{i,n})}{d_K}\right)^2.
\]
On the compact support and neighboring cells, $a_n/K\to D_i$
uniformly, and the difference quotients converge uniformly to $\psi'$.
The potential term $d_K\sum_n W_n|\psi(y_{i,n})|^2$ converges by
\eqref{eq:local-fixed-window-expansions}. Hence
\begin{equation}\label{eq:interior-recovery}
\|\mathcal S_{i,K}\psi\|_2^2\to\|\psi\|_{L^2}^2,
\qquad Q_K(\mathcal S_{i,K}\psi)\to q_i(\psi).
\end{equation}

Now fix an even cutoff function $\eta\in C_c^\infty(\mathbb R)$ satisfying $0\le\eta\le1$, $\eta\equiv1$ on $[-1,1]$, and $\operatorname{supp}\eta\subset[-2,2]$. Write $\eta_R(y):=\eta(y/R)$. If $\psi$ is a polynomial multiplied by a Gaussian, then
\[
\|(1-\eta_R)\psi\|_{L^2}+
\|y(1-\eta_R)\psi\|_{L^2}+
\|((1-\eta_R)\psi)'\|_{L^2}\longrightarrow0.
\]
For the derivative term, we use
$((1-\eta_R)\psi)'=(1-\eta_R)\psi'-R^{-1}\eta'(y/R)\psi$.
Gaussian decay makes each term tend to zero, and hence $q_i(\eta_R\psi)\to q_i(\psi)$. Below, we keep $R$ fixed as $K\to\infty$, and then let $R\to\infty$.

For the compactness statement, fix an energy bound $C_0>0$ and
consider unit vectors $z^{(K)}\in E_{i,K}$ with
$Q_K(z^{(K)})\le C_0$, extended by zero to all integer indices.
The constants below may depend on $C_0$, but are independent of
$K$ and of the particular sequence satisfying this bound.
The lower bounds in \eqref{eq:local-coercive-bounds} give,
uniformly over these vectors,
\begin{equation}\label{eq:interior-discrete-energy-control}
K\sum_{n\in\mathbb Z}|z_{n+1}^{(K)}-z_n^{(K)}|^2+
\sum_{n\in\mathbb Z}y_{i,n}^2|z_n^{(K)}|^2=O_{C_0}(1).
\end{equation}
Indeed, the negative contribution of the potential is bounded by a constant multiple of $\|z^{(K)}\|_2^2=1$, and the lower bound for $a_n$ applies to every edge on which the difference is nonzero.

Let $F_{i,K}=\mathcal I_{i,K}z^{(K)}$ be the continuous piecewise-linear interpolation with value $d_K^{-1/2}z_n^{(K)}$ at each grid point $y_{i,n}$, and let $P_{i,K}$ be the piecewise-constant interpolation with the same value on $[y_{i,n},y_{i,n+1})$. On each cell, writing $y=y_{i,n}+td_K$, $0\le t\le1$, gives
\[
F_{i,K}(y)=d_K^{-1/2}\bigl((1-t)z_n^{(K)}+tz_{n+1}^{(K)}\bigr).
\]

Integration in $t$ followed by summation over the cells gives
\begin{equation}\label{eq:interior-interpolation-identities}
\begin{aligned}
\|F_{i,K}\|_{L^2}^2
 &=\|z^{(K)}\|_2^2-\frac16\sum_n|z_{n+1}^{(K)}-z_n^{(K)}|^2,\\
\|F_{i,K}'\|_{L^2}^2
 &=K\sum_n|z_{n+1}^{(K)}-z_n^{(K)}|^2,\\
\|P_{i,K}\|_{L^2}^2&=\|z^{(K)}\|_2^2,\qquad
\|P_{i,K}-F_{i,K}\|_{L^2}^2
 =\frac13\sum_n|z_{n+1}^{(K)}-z_n^{(K)}|^2.
\end{aligned}
\end{equation}
Because the vectors are extended by zero, these formulas include the edges at the boundary of the support. Consequently,
$\|F_{i,K}\|_{L^2}^2=1+O_{C_0}(K^{-1})$,
$\|F_{i,K}'\|_{L^2}=O_{C_0}(1)$, and
$\|P_{i,K}-F_{i,K}\|_{L^2}=O_{C_0}(K^{-1/2})$.

Convexity gives
\[
|F_{i,K}(y)|^2\le d_K^{-1}
\bigl((1-t)|z_n^{(K)}|^2+t|z_{n+1}^{(K)}|^2\bigr).
\]
Within the same cell, $y^2\le2y_{i,n}^2+2d_K^2$ and
$y^2\le2y_{i,n+1}^2+2d_K^2$. Integrating over cells and summing,
\begin{equation}\label{eq:interior-interpolated-moment}
\int_{\mathbb R}y^2|F_{i,K}(y)|^2\,dy
\le 2\sum_n(y_{i,n}^2+d_K^2)|z_n^{(K)}|^2=O_{C_0}(1).
\end{equation}
The same bound holds for $P_{i,K}$. Fix $M>0$, depending only on
$C_0$ and the fixed rate functions, that bounds both moments for
all sufficiently large $K$ and all vectors under consideration.
In particular, $M$ is independent of the interval radius $R$ used below.

Thus $F_{i,K}$ is uniformly bounded in $H^1(\mathbb R)$. Weak compactness and the Rellich embedding theorem on bounded intervals give a subsequence converging to $\Psi$ weakly in $H^1(\mathbb R)$ and strongly in $L^2(-R,R)$ for every fixed finite $R$. Moreover, \eqref{eq:interior-interpolated-moment} gives
\[
\int_{|y|>R}|F_{i,K}(y)|^2\,dy\le M/R^2.
\]
Passing to the limit first on bounded intervals gives
$\int y^2|\Psi|^2\,dy\le M$, so $\Psi$ satisfies the same
tail bound. Hence
\[
\limsup_{K\to\infty}\|F_{i,K}-\Psi\|_{L^2(\mathbb R)}^2\le 4M/R^2.
\]
Letting $R\to\infty$ gives strong convergence on the whole real line. The norm identities imply $\|\Psi\|_{L^2}=1$, and the moment bound gives
$\Psi\in H^1(\mathbb R)\cap L^2(\mathbb R,y^2dy)$. We have therefore proved
\begin{equation}\label{eq:interior-compactness}
\mathcal I_{i,K}z^{(K)}\to\Psi\text{ in }L^2(\mathbb R),\qquad
\mathcal I_{i,K}z^{(K)}\rightharpoonup\Psi\text{ in }H^1(\mathbb R),
\quad\|\Psi\|_{L^2}=1.
\end{equation}
Moreover $P_{i,K}\to\Psi$ strongly in $L^2(\mathbb R)$ by \eqref{eq:interior-interpolation-identities}.

To compare the energies, choose a constant $M_i>0$, independent of $K$, such that
$W_n+M_i\ge0$ on $J_{i,K}$ for all sufficiently large $K$ and $p_i(y)+M_i\ge0$ on $\mathbb R$. Such a choice is possible by \eqref{eq:local-coercive-bounds}. Define the piecewise-constant coefficients
$A_{i,K}(y)=a_n/K$, $B_{i,K}(y)=W_n+M_i$ on intervals
$[y_{i,n},y_{i,n+1})$ for $n\ge1$, and set them
to zero on the remaining cells, where both interpolants vanish
for large $K$. Then
\[
Q_K(z^{(K)})+M_i\|z^{(K)}\|_2^2
=\int_{\mathbb R}A_{i,K}|F_{i,K}'|^2\,dy
 +\int_{\mathbb R}B_{i,K}|P_{i,K}|^2\,dy.
\]
The gradient integrand is nonnegative. The potential integrand is
also nonnegative, since $P_{i,K}$ can be nonzero only on cells
whose left grid point has index in $J_{i,K}$.

For each fixed $R$, \eqref{eq:local-fixed-window-expansions} gives
\[
A_{i,K}\to D_i,\qquad B_{i,K}\to p_i+M_i
\quad\hbox{uniformly on }[-R,R].
\]
The uniform gradient bound controls the error when $A_{i,K}$ is replaced by $D_i$. Weak lower semicontinuity then yields
\[
\liminf_{K\to\infty}\int_{-R}^R A_{i,K}|F_{i,K}'|^2\,dy
\ge D_i\int_{-R}^R|\Psi'|^2\,dy.
\]
For the potential term, strong $L^2$ convergence gives
\[
\bigl\||P_{i,K}|^2-|\Psi|^2\bigr\|_{L^1(-R,R)}
\le (\|P_{i,K}\|_{L^2}+\|\Psi\|_{L^2})\|P_{i,K}-\Psi\|_{L^2}\to0.
\]
Since $p_i+M_i$ is bounded over $[-R,R]$,
\[
\int_{-R}^R B_{i,K}|P_{i,K}|^2\,dy
\longrightarrow\int_{-R}^R(p_i+M_i)|\Psi|^2\,dy.
\]
Discarding the nonnegative contributions from $|y|>R$ and taking $K\to\infty$ followed by $R\to\infty$, we obtain
\[
\liminf_{K\to\infty}Q_K(z^{(K)})+M_i
\ge q_i(\Psi)+M_i\|\Psi\|_{L^2}^2.
\]
Since $\|\Psi\|_{L^2}=1$, the constant terms cancel, giving
\begin{equation}\label{eq:interior-form-liminf}
\liminf_{K\to\infty}Q_K(z^{(K)})\ge q_i(\Psi).
\end{equation}
The shift by $M_i$ also covers the stable oscillator, whose potential contains a negative constant term.

It remains to check convergence of inner products. Write $\Delta z_n:=z_{n+1}-z_n$. For two sequences
$z^{(K)},w^{(K)}\in E_{i,K}$ with bounded norms and energies, we have
$\sum|\Delta z^{(K)}|^2+\sum|\Delta w^{(K)}|^2=O(K^{-1})$.
Here and in the following correction estimate, the implied constants
depend only on the prescribed norm and energy bounds and the fixed rates.
Polarizing the first identity in
\eqref{eq:interior-interpolation-identities} gives
\[
\langle\mathcal I_{i,K}z^{(K)},\mathcal I_{i,K}w^{(K)}\rangle_{L^2}
=\langle z^{(K)},w^{(K)}\rangle
-\frac16\sum_n\Delta z_n^{(K)}\Delta w_n^{(K)}.
\]
Cauchy--Schwarz shows that the correction term is $O(K^{-1})$. If the interpolants converge strongly to $Z,W$, then
$\langle z^{(K)},w^{(K)}\rangle\to\langle Z,W\rangle_{L^2}$.
For a fixed smooth compactly supported function $\psi$, standard interpolation estimates also give
$\mathcal I_{i,K}\mathcal S_{i,K}\psi\to\psi$ in $L^2$. These facts allow us to project trial vectors and pass orthogonality to the limit in the estimates for the second local eigenvalue.
\end{proof}

\subsection{The local bounds}
\label{subsec:local-bounds}

We begin with the left boundary layer, where the per-capita rates are compared
with their values at the absorbing boundary. Set
\[
\alpha_0:=\widetilde\lambda(0),\qquad
\beta_0:=\widetilde\mu(0),\qquad
r_0:=\alpha_0/\beta_0\in(0,1).
\]
Recall that $\gamma_0=\beta_0-\alpha_0>0$.

\begin{lemma}[Left boundary layer]\label{lem:left-boundary-coercive}
For every $\varepsilon\in(0,\gamma_0)$, there exist
$\delta_\varepsilon\in(0,x_1/2)$ and $K_\varepsilon\ge1$ such that
\[
Q_K(z)\ge(\gamma_0-\varepsilon)\|z\|_2^2
\]
holds whenever $0<\delta\le\delta_\varepsilon$, $K\ge K_\varepsilon$, and
$\operatorname{supp}z\subset J_{b,K}$.
In particular, the estimate applies to $z=\chi_{b,K}\zeta$.
\end{lemma}

\begin{proof}
The frozen killed generator is
\[
(\mathcal L_b^{(0)}u)_n
=\alpha_0n(u_{n+1}-u_n)+\beta_0n(u_{n-1}-u_n),\qquad u_0=0.
\]
Its reversible weights may be taken as
$\pi_n^{(0)}=(\alpha_0/\beta_0)^{n-1}/n$.
Introduce the unitary transformation $(S_0u)_n=\sqrt{\pi_n^{(0)}}u_n$,
set $\mathcal H_b^{(0)}=S_0(-\mathcal L_b^{(0)})S_0^{-1}$,
and denote its closed quadratic form by $Q_b^{(0)}$.
Here we take the Friedrichs realization obtained by closing the
corresponding quadratic form on $c_{00}(\mathbb N^*)$;
see \cite[Theorem~2.13]{Teschl2014}.

The coefficients of the form are
\[
a_n^{(0)}=\sqrt{\alpha_0\beta_0n(n+1)},\qquad
W_n^{(0)}=n(\alpha_0+\beta_0)-a_n^{(0)}-a_{n-1}^{(0)},\quad a_0^{(0)}=0.
\]
The sequence $u_n=n$ solves $-\mathcal L_b^{(0)}u=\gamma_0u$.
Thus
\[
\phi_n:=\sqrt{\pi_n^{(0)}}n
=\sqrt n(\alpha_0/\beta_0)^{(n-1)/2}
\]
satisfies the transformed recurrence. To verify that it belongs to the operator domain, first note that
\[
\sum_{n\ge1}n\phi_n^2=\sum_{n\ge1}n^2r_0^{n-1}<\infty.
\]
Concavity of the square root gives, for each $n\ge1$,
\[
\sqrt{n(n+1)}+\sqrt{n(n-1)}\le2n.
\]
The frozen coefficients therefore satisfy
\[
\begin{aligned}
\sqrt{\alpha_0\beta_0}\,n
&\le a_n^{(0)}\le\sqrt{2\alpha_0\beta_0}\,n,\\
(\sqrt{\beta_0}-\sqrt{\alpha_0})^2n
&\le W_n^{(0)}\le(\alpha_0+\beta_0)n.
\end{aligned}
\]

Let $\phi^{(N)}=\mathbf1_{\{n\le N\}}\phi$. For $M>N$, the coefficient bounds and $(u-v)^2\le2u^2+2v^2$ give
\[
\|\phi^{(M)}-\phi^{(N)}\|_2^2
+Q_b^{(0)}(\phi^{(M)}-\phi^{(N)})
=O\!\left(\sum_{n>N}n\phi_n^2\right)\longrightarrow0
\quad(N\to\infty),
\]
uniformly in $M>N$, with an implied constant depending only on
$\alpha_0$ and $\beta_0$.
The edges crossing the truncation contribute $a_N^{(0)}\phi_{N+1}^2$ on the left and $a_M^{(0)}\phi_M^2$ on the right; both are controlled by the same tail series. Thus $\{\phi^{(N)}\}$ is Cauchy in the form norm. Since its $\ell^2$ limit is $\phi$, we have $\phi\in D(Q_b^{(0)})$.

Write $Q_b^{(0)}(\cdot,\cdot)$ for the associated bilinear form. Testing the recurrence against a finitely supported vector $z$ gives
\[
Q_b^{(0)}(\phi,z)=\gamma_0\langle\phi,z\rangle.
\]
Cauchy--Schwarz in the form norm extends this identity to all $z\in D(Q_b^{(0)})$. The representation theorem for closed quadratic forms
\cite[Theorem~2.14]{Teschl2014} then gives
$\phi\in D(\mathcal H_b^{(0)})$ and
$\mathcal H_b^{(0)}\phi=\gamma_0\phi$.

We can now apply the ground-state transform. For a finitely supported vector $z_n=\phi_nh_n$, the recurrence gives
\[
(n(\alpha_0+\beta_0)-\gamma_0)\phi_n
=a_n^{(0)}\phi_{n+1}+a_{n-1}^{(0)}\phi_{n-1},
\]
where the term with index $n-1$ vanishes at $n=1$. Substituting into the diagonal part of $Q_b^{(0)}(z)-\gamma_0\|z\|_2^2$ and grouping the contributions from adjacent edges, we obtain
\begin{equation}\label{eq:left-boundary-ground-state-transform}
Q_b^{(0)}(z)-\gamma_0\|z\|_2^2
=\sum_{n\ge1}a_n^{(0)}\phi_n\phi_{n+1}(h_{n+1}-h_n)^2\ge0.
\end{equation}
Density in the form norm extends this lower bound to the whole form domain. Taking $\phi$ as a trial vector gives
\begin{equation}\label{eq:left-boundary-frozen-bottom}
\inf_{\substack{z\in D(Q_b^{(0)})\\\|z\|_2=1}}Q_b^{(0)}(z)=\gamma_0.
\end{equation}
The coefficient bounds also give
\begin{equation}\label{eq:left-boundary-weighted-H1-control}
\sum_{n\ge1}n(z_{n+1}-z_n)^2+\sum_{n\ge1}nz_n^2
\le\max\left\{
\frac1{\sqrt{\alpha_0\beta_0}},
\frac1{(\sqrt{\beta_0}-\sqrt{\alpha_0})^2}
\right\}Q_b^{(0)}(z).
\end{equation}

To compare the frozen and original forms, fix $r_*>0$ so that the
per-capita rates are bounded above and bounded away from zero on
$[0,2r_*]$. For $0<r\le r_*$, define
\[
\omega(r)=\sup_{0\le x\le2r}
\bigl(|\widetilde\lambda(x)-\alpha_0|+|\widetilde\mu(x)-\beta_0|\bigr).
\]
Continuity of the rates implies $\omega(r)\to0$ as $r\downarrow0$.

After choosing $\delta_\varepsilon\le r_*$ below, we take $K_\varepsilon$ large enough that
$K^{-1/4}+K^{-1}\le\delta_\varepsilon$ for all $K\ge K_\varepsilon$.
Whenever $\delta\le\delta_\varepsilon$, every site in $J_{b,K}$ and
its right neighbor have density coordinates at most $2\delta_\varepsilon$.
There is a constant $C>0$, depending only on the rate bounds on
$[0,2r_*]$ and on $\alpha_0,\beta_0$, such that
\[
|a_n-a_n^{(0)}|+|W_n-W_n^{(0)}|
\le C\omega(\delta_\varepsilon)n
\]
on all sites and edges contributing to the quadratic form of vectors
supported in $J_{b,K}$. In particular, $C$ is independent of
$\delta$, $\delta_\varepsilon$, $K$, and $z$.

Together with \eqref{eq:left-boundary-weighted-H1-control}, this gives
\[
|Q_K(z)-Q_b^{(0)}(z)|
\le C_b\omega(\delta_\varepsilon)Q_b^{(0)}(z).
\]
Here $C_b>0$ is fixed independently of $\delta$, $\delta_\varepsilon$,
$K$, and $z$. Choose
$\delta_\varepsilon\in(0,\min\{r_*,x_1/2\})$ so that
$C_b\omega(\delta_\varepsilon)\gamma_0\le\varepsilon$. Then
\[
Q_K(z)\ge(1-C_b\omega(\delta_\varepsilon))Q_b^{(0)}(z)
\ge(\gamma_0-\varepsilon)\|z\|_2^2.
\]
Since $\delta_\varepsilon$ is fixed in advance, this same $K_\varepsilon$ works for every $0<\delta\le\delta_\varepsilon$.
\end{proof}

We next apply \resultref{Lemma~}{lem:interior-approximation} to the window around
the unstable threshold.

\begin{lemma}[Window near $x_1$]\label{lem:x1-window-precise}
The principal eigenvalue associated with $Q_{1,K}$,
\[
\lambda_{0,K}^{(1)}:=\inf_{\substack{z\in E_{1,K}\\\|z\|_2=1}}Q_K(z),
\]
converges to $\gamma_1=V'(x_1)$. Consequently, for every
$\varepsilon\in(0,\gamma_1)$ and all sufficiently large $K$,
\[
Q_K(z)\ge(\gamma_1-\varepsilon)\|z\|_2^2,\qquad z\in E_{1,K}.
\]
\end{lemma}

\begin{proof}
By \eqref{eq:interior-limit-form}, the limiting local operator is
\[
\mathcal H_{1,\mathrm{loc}}
=-D_1\partial_{yy}+\frac{\gamma_1^2}{4D_1}y^2+\frac{\gamma_1}{2}.
\]
The change of variables $\xi=\sqrt{\gamma_1/(2D_1)}y$ puts it in the form
$(\gamma_1/2)(-\partial_{\xi\xi}+\xi^2+1)$.
The standard one-dimensional oscillator has eigenvalues $2j+1$,
$j=0,1,\ldots$; see \cite[Section~8.3, Eqs.~(8.39)--(8.43)]{Teschl2014}.
The scaling and shift above therefore give the spectrum
$\{k\gamma_1:k\ge1\}$. Its normalized ground state is
\[
\psi_{1,0}(y)=\left(\frac{\gamma_1}{2\pi D_1}\right)^{1/4}
e^{-\gamma_1y^2/(4D_1)},\qquad q_1(\psi_{1,0})=\gamma_1.
\]

Using the cutoff $\eta_R$ introduced earlier, set
\begin{equation}\label{cut-off-f}
\psi_{1,0,R}(y)=\eta_R(y)\psi_{1,0}(y).
\end{equation}
Applying \eqref{eq:interior-recovery} and the variational principle for fixed $R$, then letting $R\to\infty$, gives
\[
\limsup_{K\to\infty}\lambda_{0,K}^{(1)}
\le\lim_{R\to\infty}\frac{q_1(\psi_{1,0,R})}{\|\psi_{1,0,R}\|_{L^2}^2}
=\gamma_1.
\]

For the lower bound, take a subsequence along which the local eigenvalues converge to their liminf, and let $g_{1,K}$ be the corresponding normalized ground states. Their energies are bounded, so
\eqref{eq:interior-compactness} and
\eqref{eq:interior-form-liminf} give a normalized limit $G$ satisfying
\[
\liminf_{K\to\infty}\lambda_{0,K}^{(1)}\ge q_1(G)\ge\gamma_1.
\]
The two bounds prove convergence of the eigenvalue and the stated coercive estimate.
\end{proof}

Outside the equilibrium windows, the potential gives a lower bound
in the intermediate region. The far-right tail will be treated
separately below.

\begin{lemma}[Coercivity in the intermediate region]\label{lem:middle-region-coercivity}
For each fixed sufficiently small $\delta>0$, there exists
$c_{\mathrm{mid},\delta}>0$, independent of $K$, such that for
all sufficiently large $K$,
\[
Q_K(\chi_{0,K}\zeta)\ge
c_{\mathrm{mid},\delta}K^{1/2}\|\chi_{0,K}\zeta\|_2^2,
\qquad \zeta\in\ell^2(\mathbb N^*).
\]
\end{lemma}

\begin{proof}
Define
\[
U(x)=x(\sqrt{\widetilde\lambda(x)}-\sqrt{\widetilde\mu(x)})^2.
\]
For $x>0$, the only zeros of $U(x)$ are $x_1$ and $x_2$.
Transversality at the equilibria and strict positivity elsewhere
give a constant $c_\delta>0$ such that, on $[\delta/2,2/\delta]$,
\[
U(x)\ge c_\delta\min\{1,|x-x_1|^2,|x-x_2|^2\}.
\]
By \eqref{eq:local-potential-expansion},
$W_n=KU(n/K)+O_\delta(1)$ uniformly on this interval.
For lattice points $n\in J_{0,K}$,
\[
|n/K-x_i|\ge\ell_K/K-K^{-1}\ge\tfrac12K^{-1/4},\qquad i=1,2,
\]
when $K$ is large. Taking $c_{\mathrm{mid},\delta}:=c_\delta/8$
and using the uniform remainder bound gives
$W_n\ge c_{\mathrm{mid},\delta}K^{1/2}$ on $J_{0,K}$ for all
sufficiently large $K$. The gradient contribution is nonnegative,
so the lemma follows.
\end{proof}

\begin{lemma}[Far-right tail coercivity]\label{lem:right-tail-coercive}
Given any $B_0>0$, there exists $\delta_{B_0}>0$ so that, for
each fixed $\delta\in(0,\delta_{B_0}]$ and all sufficiently large $K$,
\[
Q_K(z)\ge B_0\|z\|_2^2,
\qquad z\in D(Q_K),\quad\operatorname{supp}z\subset J_{r,K}.
\]
In particular, choosing $B_0=\gamma_*$ gives the bound needed for
$z=\chi_{r,K}\zeta$, $\zeta\in D(Q_K)$.
\end{lemma}

\begin{proof}
Fix $\theta_0=1/2$. By assumption {\rm(A3)}, we may choose
$\delta_{B_0}$ small enough that, for all $\delta\le\delta_{B_0}$,
\begin{equation}\label{eq:right-tail-choice}
\frac{\widetilde\lambda(x)}{\widetilde\mu(x)}\le\theta_0
\quad(x\ge1/(2\delta)),\qquad
\frac1{1-\theta_0}\int_{1/(2\delta)}^\infty
\frac{dx}{x\widetilde\mu(x)}\le B_0^{-1}.
\end{equation}

For fixed $\delta$, set
\[
m_K:=\left\lceil\lfloor K/\delta\rfloor-\ell_K\right\rceil.
\]
For large $K$, $m_K\ge2$ and $(m_K-1)/K\ge1/(2\delta)$.

First let $z$ be finitely supported in $J_{r,K}$, and set
$u=S_K^{-1}z$. Then $u_n=0$ whenever $n<m_K$, so in particular
$u_{m_K-1}=0$. The killing term vanishes and the quadratic form becomes
\begin{equation}\label{eq:right-tail-dirichlet-form}
Q_K(z)=\sum_{k\ge m_K-1}\lambda_k\pi_k(u_{k+1}-u_k)^2.
\end{equation}
This sum includes the edge crossing the cutoff at index $k=m_K-1$.

For $n\ge m_K$, telescoping together with Cauchy--Schwarz gives
\[
u_n=\sum_{k=m_K-1}^{n-1}(u_{k+1}-u_k),\qquad
u_n^2\le Q_K(z)\sum_{k=m_K-1}^{n-1}\frac1{\lambda_k\pi_k}.
\]
Multiplying by $\pi_n$, summing over $n$, and interchanging the nonnegative sums gives
\[
\sum_{n\ge m_K}\pi_nu_n^2
\le Q_K(z)\sum_{k\ge m_K-1}\frac1{\lambda_k\pi_k}
\sum_{n\ge k+1}\pi_n.
\]

Monotonicity of $\widetilde\mu$ and condition \eqref{eq:right-tail-choice} yield
\[
\frac{\pi_{n+1}}{\pi_n}=\frac{\lambda_n}{\mu_{n+1}}\le\theta_0,
\qquad n\ge m_K-1.
\]
Hence $\sum_{n\ge k+1}\pi_n\le\pi_{k+1}/(1-\theta_0)$.
Detailed balance then gives
\begin{equation}\label{eq:right-tail-poincare}
\sum_{n\ge m_K}\pi_nu_n^2
\le\frac{Q_K(z)}{1-\theta_0}\sum_{j\ge m_K}\frac1{\mu_j}.
\end{equation}

Because $x\mapsto1/(x\widetilde\mu(x))$ is decreasing,
\[
\sum_{j\ge m_K}\frac1{\mu_j}
=\frac1K\sum_{j\ge m_K}\frac1{(j/K)\widetilde\mu(j/K)}
\le\int_{(m_K-1)/K}^\infty\frac{dx}{x\widetilde\mu(x)}
\le\int_{1/(2\delta)}^\infty\frac{dx}{x\widetilde\mu(x)}.
\]

Combining \eqref{eq:right-tail-choice}--\eqref{eq:right-tail-poincare} gives $\|z\|_2^2\le B_0^{-1}Q_K(z)$.

For an arbitrary $z\in D(Q_K)$ with this support property, choose $z^{(j)}\in c_{00}$ converging to $z$ in the form norm. Let $P_{m_K}$ denote multiplication by the indicator $\mathbf1_{\{n\ge m_K\}}$. Since $P_{m_K}$ has only one jump, across the edge $m_K-1$, the multiplier estimate from Section~\ref{sec:self-adjoint} gives
\[
P_{m_K}z^{(j)}\longrightarrow P_{m_K}z=z\quad\hbox{in the form norm}.
\]
The inequality above holds for each finitely supported $P_{m_K}z^{(j)}$. Continuity of the norm and energy under form-norm convergence then gives $Q_K(z)\ge B_0\|z\|_2^2$.
\end{proof}

It remains to analyze the window around the positive stable equilibrium.
Here the limiting ground-state energy is zero, so we also need the
second local eigenvalue and a comparison of the local and global
ground states.

\begin{lemma}[Local spectrum near $x_2$ and global ground-state localization]\label{lem:x2-window-kernel-gap}
Let $g_{2,K}>0$ on $J_{2,K}$ denote the normalized principal eigenvector of $Q_{2,K}$, extended by zero outside $J_{2,K}$. Write $\lambda_{0,K}^{(2)}$ and
$\lambda_{1,K}^{(2)}$ for its first two local eigenvalues. Then
\begin{enumerate}
\item[\emph{(i)}] $Q_K(g_{2,K})=\lambda_{0,K}^{(2)}\to0$.
In particular, $\rho_1(K)\to0$.
\item[\emph{(ii)}] $\lambda_{1,K}^{(2)}\to\gamma_2$.
Consequently, for every $\varepsilon\in(0,\gamma_2)$ and all sufficiently large $K$,
\[
Q_K(z)\ge(\gamma_2-\varepsilon)\|z\|_2^2,
\qquad z\in E_{2,K},\quad z\perp g_{2,K}.
\]
\item[\emph{(iii)}] The normalized global principal eigenvector is
localized in the stable window, in the sense that
\[
\sum_{\sigma\ne2}\|\chi_{\sigma,K}v_{1,K}\|_2^2\to0,
\quad \|\chi_{2,K}v_{1,K}\|_2\to1,
\quad Q_K(\chi_{2,K}v_{1,K})\to0,
\]
and
\[
\left\|g_{2,K}-
\frac{\chi_{2,K}v_{1,K}}{\|\chi_{2,K}v_{1,K}\|_2}\right\|_2\to0.
\]
In particular,
$\|g_{2,K}-\chi_{2,K}v_{1,K}\|_2\to0$ and
$\|g_{2,K}-v_{1,K}\|_2\to0$.

More precisely, for each fixed sufficiently small $\delta>0$,
\begin{equation}\label{eq:x2-quantitative-localization}
\begin{aligned}
&\sum_{\sigma\ne2}\|\chi_{\sigma,K}v_{1,K}\|_2^2
 +\|v_{1,K}-g_{2,K}\|_2^2
 +\|\chi_{2,K}g_{2,K}-v_{1,K}\|_2^2\\
&\hspace{25mm}=O_\delta\bigl(\rho_1(K)+K^{-1/2}\bigr).
\end{aligned}
\end{equation}
As a consequence,
\begin{equation}\label{eq:x2-uniform-orthogonality}
\sup_{\substack{z\in\ell^2(\mathbb N^*),\ \|z\|_2=1\\
                 z\perp v_{1,K}}}
\bigl|\langle\chi_{2,K}z,g_{2,K}\rangle\bigr|^2
=O_\delta\bigl(\rho_1(K)+K^{-1/2}\bigr).
\end{equation}
\end{enumerate}
\end{lemma}

\begin{proof}
\emph{(i)} The limiting form $q_2$ from \eqref{eq:interior-limit-form} corresponds to
\[
\mathcal H_{2,\mathrm{loc}}
=-D_2\partial_{yy}+\frac{\gamma_2^2}{4D_2}y^2-\frac{\gamma_2}{2}.
\]
Its spectrum is $\{k\gamma_2:k\ge0\}$, and its normalized ground state is
\begin{equation}\label{eq:x2-limit-ground-state}
\psi_{2,0}(y)=\left(\frac{\gamma_2}{2\pi D_2}\right)^{1/4}
e^{-\gamma_2y^2/(4D_2)},\qquad q_2(\psi_{2,0})=0.
\end{equation}

Define $\psi_{2,0,R}(y)=\eta_R(y)\psi_{2,0}(y)$ using the cutoff from \eqref{cut-off-f}. Equation~\eqref{eq:interior-recovery} and nonnegativity of $Q_K$ give
\[
0\le\limsup_{K\to\infty}\lambda_{0,K}^{(2)}
\le\frac{q_2(\psi_{2,0,R})}{\|\psi_{2,0,R}\|_{L^2}^2}\longrightarrow0
\quad(R\to\infty).
\]
Hence
\begin{equation}\label{eq:x2-local-ground-energy-small}
0\le\rho_1(K)\le\lambda_{0,K}^{(2)}=Q_K(g_{2,K})\longrightarrow0.
\end{equation}
Irreducibility of the finite-dimensional Jacobi matrix yields strict positivity of $g_{2,K}$. From
\eqref{eq:interior-compactness} and
\eqref{eq:interior-form-liminf}, every subsequential limit of its interpolation is a nonnegative unit vector lying in $\ker\mathcal H_{2,\mathrm{loc}}$. Since this kernel is one-dimensional and spanned by $\psi_{2,0}$, the whole sequence converges:
\begin{equation}\label{eq:x2-local-ground-state-convergence}
\mathcal I_{2,K}g_{2,K}\longrightarrow\psi_{2,0}
\quad\hbox{in }L^2(\mathbb R).
\end{equation}

\emph{(ii)} The normalized first excited state is
\[
\psi_{2,1}(y)=\sqrt{\frac{\gamma_2}{D_2}}\,y\psi_{2,0}(y).
\]
It is orthogonal to $\psi_{2,0}$ and has eigenvalue $\gamma_2$.

For fixed $R$, sample the cutoff version
$\psi_{2,1,R}=\eta_R\psi_{2,1}$ and set
\[
z_{K,R}=\mathcal S_{2,K}\psi_{2,1,R},\qquad
\theta_{K,R}=\langle z_{K,R},g_{2,K}\rangle.
\]
Equation~\eqref{eq:x2-local-ground-state-convergence} and the interpolation identities imply $\theta_{K,R}\to0$, since $\psi_{2,1,R}$ is odd and $\psi_{2,0}$ is even.

Both $z_{K,R}$ and $g_{2,K}$ belong to $E_{2,K}$, and the local weak eigenvalue equation is
\[
Q_K(g_{2,K},z)=\lambda_{0,K}^{(2)}\langle g_{2,K},z\rangle,
\qquad z\in E_{2,K}.
\]
This equation holds on the local subspace; the zero extension of $g_{2,K}$ need not be a global eigenvector.

The projected vector $z_{K,R}-\theta_{K,R}g_{2,K}$ is orthogonal to $g_{2,K}$ and satisfies
\[
\begin{aligned}
Q_K(z_{K,R}-\theta_{K,R}g_{2,K})
 &=Q_K(z_{K,R})-\theta_{K,R}^2\lambda_{0,K}^{(2)},\\
\|z_{K,R}-\theta_{K,R}g_{2,K}\|_2^2
 &=\|z_{K,R}\|_2^2-\theta_{K,R}^2.
\end{aligned}
\]
For each sufficiently large fixed $R$, the squared norm converges to $\|\psi_{2,1,R}\|_{L^2}^2>0$. The local variational principle then gives
\[
\limsup_{K\to\infty}\lambda_{1,K}^{(2)}
\le\frac{q_2(\psi_{2,1,R})}{\|\psi_{2,1,R}\|_{L^2}^2}.
\]
Letting $R\to\infty$ and using the Gaussian cutoff estimates from Subsection~\ref{subsec:interior-approximation} gives
$\limsup_K\lambda_{1,K}^{(2)}\le\gamma_2$.

For the lower bound, take a subsequence along which the second local
eigenvalues converge to their $\liminf$, and choose corresponding
normalized eigenvectors $w_{2,K}\perp g_{2,K}$. The upper bound controls
their energies, so \eqref{eq:interior-compactness} gives a normalized
limit $W$ along a further subsequence. Equation~\eqref{eq:x2-local-ground-state-convergence}
and the convergence of inner products proved above imply
\[
\langle W,\psi_{2,0}\rangle_{L^2}
=\lim_{K\to\infty}\langle w_{2,K},g_{2,K}\rangle=0.
\]
The lower-semicontinuity estimate \eqref{eq:interior-form-liminf} and the spectral gap of the limiting oscillator give
\[
\liminf_{K\to\infty}\lambda_{1,K}^{(2)}\ge q_2(W)\ge\gamma_2.
\]
This completes the proof of (ii). The coercive estimate follows by restricting the variational minimization to the orthogonal complement of the local ground state.

\emph{(iii)} Fix $\delta>0$ small enough that the boundary estimate
applies with $\varepsilon=\gamma_0/2$ and the tail estimate applies
with $B_0=\gamma_*$. For all sufficiently large $K$, the local forms
outside the stable window have the common lower bound
$(\gamma_*/2)\|z\|_2^2$. The discrete IMS inequality
\eqref{eq:discrete-IMS} therefore gives
\[
\frac{\gamma_*}{2}\sum_{\sigma\ne2}\|\chi_{\sigma,K}v_{1,K}\|_2^2
   +Q_K(\chi_{2,K}v_{1,K})
\le E_K:=\rho_1(K)+C_{\mathrm{IMS},\delta}K^{-1/2},
\]
where $C_{\mathrm{IMS},\delta}$ is the chosen constant in
\eqref{eq:discrete-IMS}. Every term on the left-hand side is
nonnegative, and $E_K\to0$ by
\eqref{eq:x2-local-ground-energy-small}. The quadratic partition of unity gives
\[
s_K^2:=\|\chi_{2,K}v_{1,K}\|_2^2
=1-\sum_{\sigma\ne2}\|\chi_{\sigma,K}v_{1,K}\|_2^2
\ge1-\frac{2}{\gamma_*}E_K.
\]
In particular, $s_K^2\ge1/2$ for all sufficiently large $K$.

Set $u_K=\chi_{2,K}v_{1,K}/s_K$ and decompose
\[
u_K=\alpha_Kg_{2,K}+h_K,\qquad h_K\perp g_{2,K}.
\]
The local eigenvalue equation and part (ii) give
\[
\frac{\gamma_2}{2}\|h_K\|_2^2
\le\alpha_K^2\lambda_{0,K}^{(2)}+Q_K(h_K)
=Q_K(u_K)
=s_K^{-2}Q_K(\chi_{2,K}v_{1,K})\le2E_K.
\]
Positivity of the ground states and nonnegativity of the cutoff give
$\alpha_K=\langle u_K,g_{2,K}\rangle\ge0$. The normalization $\alpha_K^2+\|h_K\|_2^2=1$ then yields
\[
\|u_K-g_{2,K}\|_2^2
=2(1-\alpha_K)
=\frac{2\|h_K\|_2^2}{1+\alpha_K}
\le\frac{8}{\gamma_2}E_K.
\]
Moreover, $0\le\chi_{2,K}\le1$ and $0<s_K\le1$, so
\[
\|(1-\chi_{2,K})v_{1,K}\|_2^2\le1-s_K^2\le\frac{2}{\gamma_*}E_K,
\qquad (1-s_K)^2\le1-s_K^2\le\frac{2}{\gamma_*}E_K.
\]

The decomposition
\[
v_{1,K}-g_{2,K}
=(1-\chi_{2,K})v_{1,K}+(s_K-1)u_K+(u_K-g_{2,K}),
\]
together with the triangle and Cauchy--Schwarz inequalities, gives
$\|v_{1,K}-g_{2,K}\|_2^2=O_\delta(E_K)$.

It follows also that
\[
\begin{aligned}
\|\chi_{2,K}g_{2,K}-v_{1,K}\|_2^2
&=\|\chi_{2,K}(g_{2,K}-v_{1,K})-(1-\chi_{2,K})v_{1,K}\|_2^2\\
&\le2\|g_{2,K}-v_{1,K}\|_2^2
   +2\|(1-\chi_{2,K})v_{1,K}\|_2^2
=O_\delta(E_K).
\end{aligned}
\]
Since $E_K=O_\delta(\rho_1(K)+K^{-1/2})$, these estimates give
\eqref{eq:x2-quantitative-localization}.

For any unit vector $z\perp v_{1,K}$, multiplication by $\chi_{2,K}$ is self-adjoint, so
\[
\bigl|\langle\chi_{2,K}z,g_{2,K}\rangle\bigr|
=\bigl|\langle z,\chi_{2,K}g_{2,K}-v_{1,K}\rangle\bigr|
\le\|\chi_{2,K}g_{2,K}-v_{1,K}\|_2.
\]
This establishes \eqref{eq:x2-uniform-orthogonality}. The bounds for
$E_K$, $s_K$, $Q_K(\chi_{2,K}v_{1,K})$, and
$\|u_K-g_{2,K}\|_2$ also imply all convergence statements in item (iii).
\end{proof}
\section{Proof of the spectral-gap limit}
\label{sec:gap-proof}
\subsection{Lower bound}
\label{sec:lower-bound}
The local estimates give a lower bound once the component along the
stable local ground state has been controlled. The orthogonality
estimate from the preceding section provides this control uniformly
over the unit sphere in $v_{1,K}^{\perp}$.

\begin{proof}[Proof of \resultref{Proposition~}{prop:spectral-gap-liminf-lower}]
Fix $0<\varepsilon<\gamma_*$, and put $\varepsilon_1=\varepsilon/3$.
Choose $\delta>0$ sufficiently small so that
\resultref{Lemma~}{lem:left-boundary-coercive} holds with error $\varepsilon_1$ and
\resultref{Lemma~}{lem:right-tail-coercive} holds with $B_0=\gamma_*$.
For large $K$, coercive bounds hold on the boundary layer, threshold
window, intermediate region, and far-right tail:
\begin{equation}\label{eq:global-other-regions}
Q_K(\chi_{\sigma,K}z)\ge(\gamma_*-\varepsilon_1)\|\chi_{\sigma,K}z\|_2^2,
\qquad\sigma\in\{b,0,1,r\}.
\end{equation}

Let $z\in D(Q_K)$, $z\perp v_{1,K}$, and $\|z\|_2=1$.
The quantitative estimate \eqref{eq:x2-quantitative-localization} gives
\[
e_K:=\|\chi_{2,K}g_{2,K}-v_{1,K}\|_2,
\qquad e_K^2=O_\delta\bigl(\rho_1(K)+K^{-1/2}\bigr)=o(1).
\]
Set $b_K(z):=\langle\chi_{2,K}z,g_{2,K}\rangle$.
Self-adjointness of the cutoff multiplication and orthogonality $z\perp v_{1,K}$ give
\begin{equation}\label{eq:global-bk-small}
|b_K(z)|=|\langle\chi_{2,K}z,g_{2,K}\rangle|
=|\langle z,\chi_{2,K}g_{2,K}-v_{1,K}\rangle|\le e_K.
\end{equation}
This bound holds uniformly for all such unit vectors $z$.

Decompose
\[
\chi_{2,K}z=b_K(z)g_{2,K}+r_K(z),\qquad r_K(z)\perp g_{2,K}.
\]
Since $g_{2,K}$ is a local eigenvector and $\lambda_{0,K}^{(2)}\ge0$,
\resultref{Lemma~}{lem:x2-window-kernel-gap}(ii) gives
\begin{equation}\label{eq:global-x2-bound}
\begin{aligned}
Q_K(\chi_{2,K}z)
&=b_K(z)^2\lambda_{0,K}^{(2)}+Q_K(r_K(z))\\
&\ge(\gamma_2-\varepsilon_1)(\|\chi_{2,K}z\|_2^2-b_K(z)^2)\\
&\ge(\gamma_*-\varepsilon_1)\|\chi_{2,K}z\|_2^2-\gamma_2e_K^2.
\end{aligned}
\end{equation}

Combining \eqref{eq:global-other-regions},
\eqref{eq:global-x2-bound}, and the IMS estimate
\eqref{eq:discrete-IMS} gives
\[
Q_K(z)\ge\gamma_*-\varepsilon_1-\gamma_2e_K^2
-C_{\mathrm{IMS},\delta}K^{-1/2},
\]
where we used $\sum_\sigma\|\chi_{\sigma,K}z\|_2^2=1$.
The error $\gamma_2e_K^2+C_{\mathrm{IMS},\delta}K^{-1/2}$ is
$O_\delta(\rho_1(K)+K^{-1/2})=o(1)$ uniformly over the constrained
unit sphere, and is therefore at most
$\varepsilon_1$ for sufficiently large $K$. Homogeneity of
the quadratic form gives the stated form bound, and the variational
principle yields
$\rho_2(K)\ge\gamma_*-2\varepsilon/3$.
Since $\rho_1(K)\to0$, by taking $K$ even larger if needed we obtain
$\rho_2(K)-\rho_1(K)\ge\gamma_*-\varepsilon$.
\end{proof}

\subsection{Upper bound and conclusion}
\label{sec:upper-bound}

For the upper bounds, we project local trial states onto the orthogonal complement of the global ground state.
Given a normalized $z_K\in D(Q_K)$, set
$\theta_K=\langle z_K,v_{1,K}\rangle$ and
$\widetilde z_K=z_K-\theta_Kv_{1,K}$.
The weak eigenvalue identity $Q_K(v_{1,K},z)=\rho_1(K)\langle v_{1,K},z\rangle$ holds for every $z\in D(Q_K)$. Expanding this bilinear form gives
\begin{equation}\label{eq:trial-projection}
\|\widetilde z_K\|_2^2=1-\theta_K^2,\qquad
Q_K(\widetilde z_K)=Q_K(z_K)-\theta_K^2\rho_1(K).
\end{equation}
If $\theta_K\to0$, then $\widetilde z_K$ is nonzero for all sufficiently large $K$ and is an admissible trial vector for the variational problem on $v_{1,K}^{\perp}$.
Hence
\begin{equation}\label{eq:trial-variational-bound}
\rho_2(K)\le
\frac{Q_K(z_K)-\theta_K^2\rho_1(K)}{1-\theta_K^2}.
\end{equation}

\Needspace{6\baselineskip}
We first obtain the boundary upper bound using finitely supported
trial vectors.

\begin{proof}[Proof of \resultref{Proposition~}{prop:left-boundary-limsup-gap}]
By \eqref{eq:left-boundary-frozen-bottom} and the form-core property
of finitely supported vectors, for every $\varepsilon_1>0$ there
exists a finitely supported unit vector $z$ such that
\[
Q_b^{(0)}(z)\le \gamma_0+\varepsilon_1.
\]
On the fixed support of $z$ and its adjacent edges, the coefficients
of $Q_K$ converge to those of the frozen boundary form. Thus
\[
Q_K(z)\longrightarrow Q_b^{(0)}(z).
\]
For sufficiently large $K$, this support is disjoint from $J_{2,K}$,
so $z\perp g_{2,K}$. By
\resultref{Lemma~}{lem:x2-window-kernel-gap}(iii),
\[
|\langle z,v_{1,K}\rangle|
=
|\langle z,v_{1,K}-g_{2,K}\rangle|
\le \|v_{1,K}-g_{2,K}\|_2
\longrightarrow0.
\]
The variational estimate \eqref{eq:trial-variational-bound} therefore
gives
\[
\limsup_{K\to\infty}\rho_2(K)
\le Q_b^{(0)}(z)
\le \gamma_0+\varepsilon_1.
\]
Letting $\varepsilon_1\downarrow0$ and using $\rho_1(K)\to0$
proves the proposition.
\end{proof}

For the interior upper bounds, we use the lowest local mode near
the unstable threshold and the first excited local mode near the
stable equilibrium.

\begin{proof}[Proof of \resultref{Proposition~}{prop:spectral-gap-upper-bound-interior}]
Let $g_{1,K}$ and $w_{2,K}$ be normalized eigenvectors of the
operators associated with $Q_{1,K}$ and $Q_{2,K}$, corresponding
to $\lambda_{0,K}^{(1)}$ and $\lambda_{1,K}^{(2)}$, respectively.
By \resultref{Lemma~}{lem:x1-window-precise} and
\resultref{Lemma~}{lem:x2-window-kernel-gap}(ii),
\[
Q_K(g_{1,K})=\lambda_{0,K}^{(1)}\longrightarrow\gamma_1,
\qquad
Q_K(w_{2,K})=\lambda_{1,K}^{(2)}\longrightarrow\gamma_2.
\]
For sufficiently large $K$, the windows $J_{1,K}$ and $J_{2,K}$
are disjoint, so $g_{1,K}\perp g_{2,K}$. Moreover,
$w_{2,K}\perp g_{2,K}$ by orthogonality of the local eigenvectors.
Hence, for either $u=g_{1,K}$ or $u=w_{2,K}$,
\[
|\langle u,v_{1,K}\rangle|
=
|\langle u,v_{1,K}-g_{2,K}\rangle|
\le \|v_{1,K}-g_{2,K}\|_2
\longrightarrow0,
\]
where the convergence follows from
\resultref{Lemma~}{lem:x2-window-kernel-gap}(iii).
Applying \eqref{eq:trial-variational-bound} to these two trial
vectors yields
\[
\limsup_{K\to\infty}\rho_2(K)
\le \min\{\gamma_1,\gamma_2\}.
\]
Since $\rho_1(K)\to0$, the proposition follows.
\end{proof}

Combining \resultref{Proposition~}{prop:left-boundary-limsup-gap}
with \resultref{Proposition~}{prop:spectral-gap-upper-bound-interior},
we obtain
\[
\limsup_{K\to\infty}
\bigl(\rho_2(K)-\rho_1(K)\bigr)
\le \min\{\gamma_0,\gamma_1,\gamma_2\}
=\gamma_*.
\]
Together with the lower bound in
\resultref{Proposition~}{prop:spectral-gap-liminf-lower}, this proves
\resultref{Theorem~}{thm:spectral-gap-limit}.

\section{Consequences, examples, and discussion}
\label{sec:consequences}
\subsection{The Q-process and separation of time scales}
\label{subsec:q-process-consequences}

For fixed $K$, assumption {\rm(A3)} gives
$\pi_{n+1}^{(K)}/\pi_n^{(K)}\to0$ and
$\sum_{n\ge1}1/\mu_n^{(K)}<\infty$ by the integral test.
The tails of $\pi^{(K)}$ are therefore bounded by a geometric series, and
\[
\sum_{n\ge1}\frac1{\mu_n^{(K)}\pi_n^{(K)}}
\sum_{j\ge n}\pi_j^{(K)}<\infty.
\]
This verifies the birth--death criterion in
\cite[Theorem~4.1, Eq.~(4.1)]{Champagnat-Villemonais2016}, with zero catastrophe rates.
The positive bounded eigenfunction of the killed semigroup given by
\cite[Proposition~2.3]{Champagnat-Villemonais2016} belongs to
$\ell^2(\pi^{(K)})$, since $\sum_n\pi_n^{(K)}<\infty$.
Positivity and simplicity identify it, up to normalization, with
$\varphi_{1,K}$, and its decay rate with $\rho_1(K)$.
Thus \cite[Theorem~3.1(i)--(iii)]{Champagnat-Villemonais2016}
gives the $Q$-process as the limit of conditioning on survival to a
distant time, with the Doob semigroup below. Its invariant law is
proportional to $\varphi_{1,K}\nu^{(K)}$ and therefore equals
$m^{(K)}$, by the formula for $\nu^{(K)}$ given below.

To prove \resultref{Corollary~}{cor:q-process-gap}, let $Z_K$ denote the
normalization constant defined in Subsection~\ref{subsec:main-results}, and define
\[
\mathcal T_K:L^2(m^{(K)})\longrightarrow\ell^2(\pi^{(K)}),
\quad (\mathcal T_Kf)_n=Z_K^{-1/2}\varphi_{1,K}(n)f_n.
\]
This map is unitary, and the Doob-transformed generator is related to
the killed generator by
\[
D(\mathcal L_K^Q)=\mathcal T_K^{-1}D(L_K),\qquad
\mathcal T_K(-\mathcal L_K^Q)\mathcal T_K^{-1}=-L_K-\rho_1(K).
\]
The Dirichlet-form domain of the killed generator in
$\ell^2(\pi^{(K)})$ is carried by $\mathcal T_K^{-1}$ onto the form domain of the conditioned generator. For bounded $f$, the transformed semigroup is given by
\[
(P_t^{Q,K}f)(n)
=\frac{e^{\rho_1(K)t}}{\varphi_{1,K}(n)}
\bigl(e^{tL_K}(\varphi_{1,K}f)\bigr)(n).
\]
Since $\varphi_{1,K}f\in\ell^2(\pi^{(K)})$, the expression is well-defined.
This is the representation in
\cite[Theorem~3.1(ii), Eq.~(3.1)]{Champagnat-Villemonais2016}.
Unitary equivalence with the shifted killed generator also shows that
$P_t^{Q,K}$ is self-adjoint on $L^2(m^{(K)})$.

The same unitary equivalence gives
\[
\operatorname{Spec}(-\mathcal L_K^Q)
=\{0,\rho_2(K)-\rho_1(K),\rho_3(K)-\rho_1(K),\ldots\},
\]
where constant functions span the eigenspace at zero.
Let $\mathcal E_K^Q$ denote the Dirichlet form of
$-\mathcal L_K^Q$. For $g\in L^2(m^{(K)})$, write
\[
m^{(K)}(g):=\sum_{n\ge1}m_n^{(K)}g_n,\qquad
\operatorname{Var}_{m^{(K)}}(g)
:=\|g-m^{(K)}(g)\|_{L^2(m^{(K)})}^2.
\]
The spectral theorem gives
\[
\operatorname{Var}_{m^{(K)}}(g)
\le C_{\mathrm P}(K)\mathcal E_K^Q(g,g),
\qquad g\in D(\mathcal E_K^Q),
\]
with optimal Poincar\'e constant
\[
C_{\mathrm P}(K)=\frac1{\rho_2(K)-\rho_1(K)}\longrightarrow\frac1{\gamma_*}.
\]
Equivalently, for every $g\in L^2(m^{(K)})$ and $t\ge0$,
\[
\left\|P_t^{Q,K}g-m^{(K)}(g)\right\|_{L^2(m^{(K)})}
\le e^{-(\rho_2(K)-\rho_1(K))t}
\left\|g-m^{(K)}(g)\right\|_{L^2(m^{(K)})}.
\]
This completes the proof of \resultref{Corollary~}{cor:q-process-gap}.

For \resultref{Corollary~}{cor:two-scale-spectrum}, we combine
\resultref{Theorem~}{thm:spectral-gap-limit} with the sharp
asymptotic formula \eqref{eq:principal-eigenvalue-asymptotics}
from \cite[Theorem~A]{Hou-Yan-Zhou2026}. Since $\rho_1(K)\to0$,
\[
\rho_2(K)=\gamma_*+o(1),\qquad
\frac{\rho_2(K)-\rho_1(K)}{\rho_1(K)}\longrightarrow\infty.
\]
Recall that $\tau_0$ denotes the extinction time. The quasi-stationary distribution constructed in
\cite[Theorem~B]{Hou-Yan-Zhou2026} takes the form
\[
\nu_n^{(K)}=
\frac{\pi_n^{(K)}\varphi_{1,K}(n)}
{\sum_{j\ge1}\pi_j^{(K)}\varphi_{1,K}(j)}.
\]
From detailed balance and assumption {\rm(A3)},
$\pi_{n+1}^{(K)}/\pi_n^{(K)}\to0$, so $\sum_n\pi_n^{(K)}<\infty$.
The denominator is finite by Cauchy--Schwarz and $\varphi_{1,K}\in\ell^2(\pi^{(K)})$.

By self-adjointness, for any measurable set $\mathsf A\subset\mathbb N^*$,
\[
\mathbb P_{\nu^{(K)}}(X_t^K\in\mathsf A,\tau_0>t)
=\frac{\langle\varphi_{1,K},e^{tL_K}\mathbf1_{\mathsf A}\rangle_{\pi^{(K)}}}
{\langle\varphi_{1,K},\mathbf1\rangle_{\pi^{(K)}}}
=e^{-\rho_1(K)t}\nu^{(K)}(\mathsf A).
\]
Setting $\mathsf A=\mathbb N^*$ shows that under the
quasi-stationary law, the extinction time is exponentially distributed
with rate $\rho_1(K)$. This identifies the parameter in the general
exponential-survival property \cite[Proposition~2]{Meleard-Villemonais2012}. Consequently,
\[
\mathbb E_{\nu^{(K)}}\tau_0=\rho_1(K)^{-1}\sim A^{-1}e^{KH(x_1)}.
\]
The spectral relaxation time for the $Q$-process, defined as the inverse $L^2$ spectral gap, converges to $1/\gamma_*$.
This proves \resultref{Corollary~}{cor:two-scale-spectrum}.

A bounded spectral relaxation time as $K$ grows does not
by itself imply uniformly bounded mixing times for $K$-dependent initial
states. For example, the centered density of the point mass $\delta_n$
with respect to $m^{(K)}$ has $L^2$ norm
\[
\left\|\frac{\mathbf1_{\{n\}}}{m_n^{(K)}}-1\right\|_{L^2(m^{(K)})}
=\sqrt{\frac1{m_n^{(K)}}-1},
\]
which may diverge as $K\to\infty$. Bounds on mixing times from a given
initial distribution therefore also require control of its distance
to equilibrium.

\subsection{Examples}
\label{sec:logistic-example}
We first examine a logistic model in which the threshold determines
the gap and changes in the birth and death rates affect extinction
and spectral relaxation differently. We then vary the local rates
while keeping the positive equilibria fixed.

Consider the logistic birth--death model with per-capita rates
\[
    \widetilde\lambda(x)
    =
    r(a+b)x+s,
    \qquad
    \widetilde\mu(x)
    =
    rx^2+rab+s,
    \qquad
    r,s>0,\quad 0<a<b.
\]
Assumptions {\rm(A1)}--{\rm(A3)} and {\rm(H)} are all satisfied. Indeed,
\[
\left(\ln\frac{\widetilde\mu}{\widetilde\lambda}\right)'(x)
=\frac{r\{r(a+b)(x^2-ab)+s(2x-a-b)\}}
       {\widetilde\lambda(x)\widetilde\mu(x)}<0,
\qquad 0<x<a.
\]
The per-capita death rate grows quadratically, and the logarithmic ratio has the required decay properties at large $x$.
The two positive equilibria are $x_1=a$ and $x_2=b$, so the local rates are
\[
    \gamma_0=rab,
    \qquad
    \gamma_1=ra(b-a),
    \qquad
    \gamma_2=rb(b-a).
\]
Since $0<a<b$, we have $\gamma_1<\gamma_0$ and
$\gamma_1<\gamma_2$. Thus the threshold rate determines the limiting spectral gap throughout this family.
By \resultref{Theorem~}{thm:spectral-gap-limit},
\[
    \lim_{K\to\infty}
    \bigl(\rho_2(K)-\rho_1(K)\bigr)
    =
    ra(b-a)
    =
    V'(a).
\]
Changing $s$ adds the same amount to the per-capita birth and death
rates and leaves the deterministic drift unchanged. The limiting gap
is therefore unchanged, whereas the logarithmic potential $H$ and
the principal-eigenvalue asymptotics generally vary with $s$.
This family separates the dependence of spectral relaxation on the
local drift derivatives from the dependence of extinction on the
potential barrier.

For fixed $r,b>0$, we have
\[
\gamma_*=ra(b-a)
=r\left[\frac{b^2}{4}-\left(a-\frac b2\right)^2\right].
\]
Thus, as $a$ ranges over $(0,b)$, $\gamma_*$ takes values in $(0,rb^2/4]$ and attains its maximum at $a=b/2$.
Figure~\ref{fig:logistic-rates} plots the three local rates for
$r=s=1$ and $b=2$.

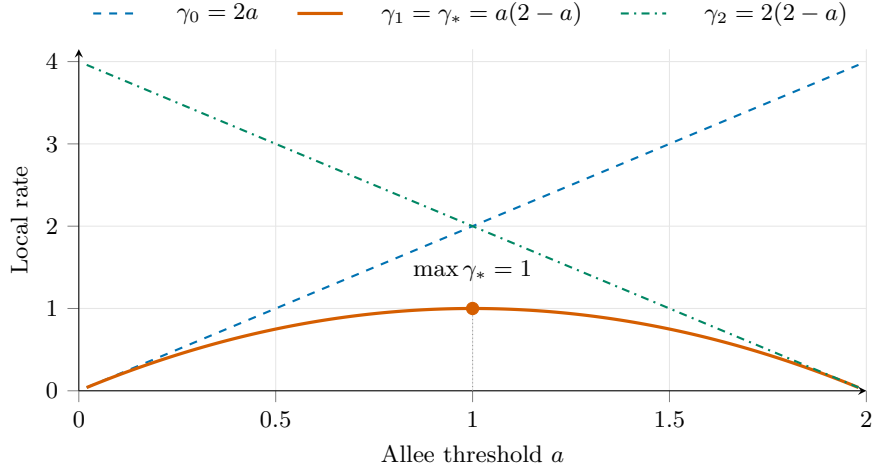
\begin{figure}[!tbp]
\centering
\begin{tikzpicture}
\begin{axis}[
    width=12cm,height=6.1cm,
    xmin=0,xmax=2,ymin=0,ymax=4.15,
    xlabel={Allee threshold $a$},ylabel={Local rate},
    xtick={0,0.5,1,1.5,2},ytick={0,1,2,3,4},
    axis lines=left,tick align=outside,
    grid=major,grid style={black!10},
    label style={font=\small},tick label style={font=\small},
    legend style={at={(0.5,1.04)},anchor=south,draw=none,
      legend columns=3,font=\small,column sep=12pt},
    clip=false]
\addplot[gapBoundary,thick,dashed,domain=0.02:1.98,samples=101] {2*x};
\addlegendentry{$\gamma_0=2a$}
\addplot[gapThreshold,very thick,domain=0.02:1.98,samples=121] {x*(2-x)};
\addlegendentry{$\gamma_1=\gamma_*=a(2-a)$}
\addplot[gapStable,thick,dashdotted,domain=0.02:1.98,samples=101] {2*(2-x)};
\addlegendentry{$\gamma_2=2(2-a)$}
\draw[black!40,densely dotted] (axis cs:1,0)--(axis cs:1,1);
\addplot[only marks,mark=*,mark size=2.3pt,gapThreshold,forget plot]
  coordinates {(1,1)};
\node[anchor=south,font=\small] at (axis cs:1,1.22)
  {$\max\gamma_*=1$};
\end{axis}
\end{tikzpicture}
\caption{Local rates and the limiting principal spectral gap in the
logistic model, with $r=s=1$, $b=2$, and $0<a<2$.
The solid curve represents $\gamma_* = \gamma_1=a(2-a)$,
which attains its maximum $1$ at $a=1$. The curves are given by the exact formulas for the local rates in
\resultref{Theorem~}{thm:spectral-gap-limit}.}
\label{fig:logistic-rates}
\end{figure}

The dominance of the threshold rate $\gamma_1$ in the logistic family
does not extend to all models. Each of the three local rates can be
the unique minimum, even when the positive equilibria are fixed, as
the following proposition shows.

\begin{proposition}[Each local rate can determine the gap]
\label{prop:realization-three-mechanisms}
Fix $0<a<b$. For each $j\in\{0,1,2\}$, there exist smooth positive
rate functions $\widetilde\lambda,\widetilde\mu$ satisfying
{\rm(A1)}--{\rm(A3)} and {\rm(H)}, with $x_1=a$ and $x_2=b$, such that
\[
\gamma_j<\min_{\substack{i\in\{0,1,2\}\\i\ne j}}\gamma_i.
\]
\end{proposition}

\begin{proof}
Choose $h\in C^3([0,\infty))$ satisfying
\[
\begin{cases}
    h(x)>0, &0\leq x<a,\\
    h(a)=0,\quad h(x)<0, &a<x<b,\\
    h(b)=0,\quad h(x)>0, &x>b,
\end{cases}
\]
together with
\[
    h'(x)<0\quad (0<x<a),
    \qquad
    h'(a)<0,
    \qquad
    h'(b)>0.
\]
We also require that there exist $R>b$, $c>0$,
$d\in\mathbb R$ such that
\[
    h(x)=cx+d,
    \qquad x\geq R.
\]
Choose $M>\|h'\|_\infty$, and set
\[
    \widetilde\lambda(x)=e^{Mx},
    \qquad
    \widetilde\mu(x)=e^{Mx+h(x)}.
\]
Then
\[
    \widetilde\lambda'(x)=Me^{Mx}>0,
    \qquad
    \widetilde\mu'(x)
    =
    \bigl(M+h'(x)\bigr)e^{Mx+h(x)}>0,
\]
and
\[
    \ln\frac{\widetilde\mu(x)}
              {\widetilde\lambda(x)}
    =
    h(x).
\]
The signs of $h$ and its derivatives imply
${\rm(A1)}$--${\rm(A2)}$, with equilibria $x_1=a$ and $x_2=b$.
Since $h$ is affine for large $x$,
\[
    \frac{\widetilde\lambda(x)}
         {\widetilde\mu(x)}
    =
    e^{-h(x)}
    \longrightarrow0,
    \qquad
    \frac{\widetilde\mu'(x)}
         {\widetilde\mu(x)}
    =
    M+h'(x),
\]
and
\[
    \int_b^\infty
    \frac{dx}{x\widetilde\mu(x)}
    =
    \int_b^\infty
    \frac{e^{-Mx-h(x)}}{x}\,dx
    <\infty.
\]
Thus ${\rm(A3)}$ holds.
Moreover,
\[
    H'(x)
    =
    h(x),
    \qquad
    H'''(x)
    =
    h''(x).
\]
Since $h$ is affine for large $x$, $h''$ has compact support, and therefore
\[
    \sup_{x\geq0}
    (1+x^2)|H'''(x)|
    <\infty.
\]
Thus ${\rm(H)}$ also holds.

The three local rates are
\[
    \gamma_0
    =
    e^{h(0)}-1,
    \qquad
    \gamma_1
    =
    -ae^{Ma}h'(a),
    \qquad
    \gamma_2
    =
    be^{Mb}h'(b).
\]
All three are strictly positive. We construct $h$ so that any one of
them can be the unique minimum while keeping $M$ fixed.

Choose smooth cutoffs $\theta:[a,b]\to[0,1]$ and
$\kappa:[b,\infty)\to[0,1]$: $\theta=1$ near $a$,
$\theta=0$ near $b$, $\kappa=1$ near $b$, and
$\kappa=0$ on $[b+1,\infty)$.
Let $w\in C^\infty([0,a])$ satisfy $0<w\le1$, taking value
$p\in(0,1]$ near $a$, and choose $q\in(0,1]$.
Define
\[
h(x)=
\begin{cases}
\displaystyle\int_x^a w(t)\,dt,&0\le x\le a,\\[2mm]
-p(x-a)\theta(x)-q(b-x)(1-\theta(x)),&a\le x\le b,\\[1mm]
(x-b)\bigl(q\kappa(x)+1-\kappa(x)\bigr),&x\ge b.
\end{cases}
\]
The formulas define the same linear function on both sides of each
junction, so $h$ is smooth. It has the required signs, and $h'=-w<0$ on $(0,a)$,
$h'(a)=-p$, $h'(b)=q$, and $h(x)=x-b$ for $x\ge b+1$. Moreover,
\[
\|h'\|_\infty\le
C_h:=1+(b-a)\|\theta'\|_\infty+\|\kappa'\|_\infty,
\]
where the bound does not depend on $p,q,w$. We fix $M>C_h$ once and for all.

To obtain the three cases, choose smooth functions $\eta_{\mathrm{ex}}$ on
$[0,\infty)$ and $\omega_{\mathrm{ex}}$ on $[0,a]$, both mapping into
$[0,1]$, such that $\eta_{\mathrm{ex}}=1$ near zero, $\eta_{\mathrm{ex}}=0$ on
$[1,\infty)$, $\omega_{\mathrm{ex}}=1$ near zero, and $\omega_{\mathrm{ex}}=0$ near
$a$. The following choices give the three cases, with limits taken
as $\varepsilon\downarrow0$.
\begin{enumerate}
\item To make the boundary rate the unique minimum, set
$w_\varepsilon(x)=\varepsilon+(1-\varepsilon)
\eta_{\mathrm{ex}}((a-x)/\varepsilon)$ and $q=1$.
Then $p=1$ and $0<h(0)\le(a+1)\varepsilon$.
Thus $\gamma_0\to0$, while $\gamma_1=ae^{Ma}$ and
$\gamma_2=be^{Mb}$ remain unchanged and positive.
\item To make the threshold rate the unique minimum, set
$w_\varepsilon=\varepsilon+(1-\varepsilon)\omega_{\mathrm{ex}}$ and $q=1$.
Then $p=\varepsilon$, and
$h(0)\to\int_0^a\omega_{\mathrm{ex}}(t)\,dt>0$.
Hence $\gamma_1=ae^{Ma}\varepsilon\to0$, and the other two
rates stay bounded away from zero.
\item To make the stable-equilibrium rate the unique minimum, set $w\equiv1$ and
$q=\varepsilon$. Then $p=1$, $h(0)=a$, and
$\gamma_2=be^{Mb}\varepsilon\to0$, while $\gamma_0$ and
$\gamma_1$ remain fixed and positive.
\end{enumerate}
For each sufficiently small fixed $\varepsilon>0$, the resulting smooth
rates satisfy all standing assumptions, and the chosen rate is the
unique minimum. Assumption {\rm(H)} requires a finite bound for each
model, not a uniform bound over the parameter family. We fix
$\varepsilon$ before taking $K\to\infty$.
\end{proof}

An explicit two-parameter family can be obtained using the following
$C^3$ interpolation, which gives examples for all three cases in
\resultref{Proposition~}{prop:realization-three-mechanisms}.
Define
\[
\mathfrak s(t)=
\begin{cases}
0,&t\le0,\\
35t^4-84t^5+70t^6-20t^7,&0<t<1,\\
1,&t\ge1.
\end{cases}
\]
This $C^3(\mathbb R)$ function satisfies
$\mathfrak s(1-t)=1-\mathfrak s(t)$ and $0\le\mathfrak s'(t)\le35/16$. Set $a=1/10$, $b=1/5$,
and $M=10$. In the piecewise definition of $h$ above, take
\[
\begin{aligned}
\theta(x)&=1-\mathfrak s\left(\frac{2(x-a)}{b-a}-\frac12\right),
&\kappa(x)&=1-\mathfrak s\left(2(x-b)-\frac12\right),\\
\omega_{\mathrm{ex}}(x)&=1-\mathfrak s\left(\frac{2x}{a}-\frac12\right),
&w_p(x)&=p+(1-p)\omega_{\mathrm{ex}}(x),
\end{aligned}
\]
with $0<p,q\le1$, and let $w=w_p$.
The cutoffs are constant near the junctions, so $h\in C^3$. Moreover,
\[
\|h'\|_\infty\le 1+\frac{35}{8}+\frac{35}{8}
=\frac{39}{4}<M.
\]
Hence $\widetilde\lambda=e^{10x}$ and
$\widetilde\mu=e^{10x+h(x)}$ are strictly increasing.
The sign and tail conditions in the preceding proof remain valid.
Thus {\rm(A1)}--{\rm(A3)} and {\rm(H)} hold, since $C^3$ regularity is sufficient for these assumptions.
By the symmetry of $\mathfrak s$,
$\int_0^a\omega_{\mathrm{ex}}(x)\,dx=a/2$, giving
\[
\gamma_0=e^{(1+p)/20}-1,\qquad
\gamma_1=\frac{e}{10}p,\qquad
\gamma_2=\frac{e^2}{5}q.
\]

Figure~\ref{fig:three-mechanisms} compares local rates for
\[
A:(p,q)=(0.4,0.1),\qquad
B:(p,q)=(0.1,0.1),\qquad
C:(p,q)=(0.4,0.02),
\]
for which the unique minimum is attained by $\gamma_0$, $\gamma_1$, and
$\gamma_2$, respectively. The parameter plot shows the regions where each rate is minimal. Two rates coincide along the separating curves, and all three coincide at their intersection.

\begin{figure}[!tbp]
\centering
\begin{tikzpicture}
\begin{axis}[
    name=bars,width=7.3cm,height=6.0cm,
    xmin=0.5,xmax=3.5,ymin=0,ymax=0.175,
    xtick={1,2,3},
    xticklabels={\shortstack{$A$\\$\gamma_*=0.0725$},
                 \shortstack{$B$\\$\gamma_*=0.0272$},
                 \shortstack{$C$\\$\gamma_*=0.0296$}},
    ytick={0,0.05,0.10,0.15},
    ylabel={Local rate},
    scaled y ticks=false,
    yticklabel style={/pgf/number format/fixed,/pgf/number format/precision=2},
    title={(a) Three parameter choices},title style={font=\small,yshift=17pt},
    label style={font=\small},tick label style={font=\scriptsize},
    axis lines=left,tick align=outside,
    ymajorgrids=true,grid style={black!10},
    legend style={at={(0.5,1.01)},anchor=south,draw=none,
      legend columns=3,font=\small,column sep=4pt},clip=false]
\addplot[ybar,area legend,bar width=8pt,bar shift=-10pt,fill=gapBoundary!75,draw=gapBoundary] coordinates {(1,0.072508181254) (2,0.056540614675) (3,0.072508181254)};
\addlegendentry{$\gamma_0$}
\addplot[ybar,area legend,bar width=8pt,bar shift=0pt,fill=gapThreshold!75,draw=gapThreshold] coordinates {(1,0.108731273138) (2,0.027182818285) (3,0.108731273138)};
\addlegendentry{$\gamma_1$}
\addplot[ybar,area legend,bar width=8pt,bar shift=10pt,fill=gapStable!75,draw=gapStable] coordinates {(1,0.147781121979) (2,0.147781121979) (3,0.029556224396)};
\addlegendentry{$\gamma_2$}
\node[font=\scriptsize,anchor=south] at ([xshift=-10pt]axis cs:1,0.072508181254)
  {$\star$};
\node[font=\scriptsize,anchor=south] at (axis cs:2,0.027182818285)
  {$\star$};
\node[font=\scriptsize,anchor=south] at ([xshift=10pt]axis cs:3,0.029556224396)
  {$\star$};
\end{axis}
\begin{axis}[
    at={(bars.east)},anchor=west,xshift=1.25cm,
    width=7.2cm,height=6.0cm,
    xmin=0.02,xmax=0.5,ymin=0.005,ymax=0.13,
    xlabel={$p$},ylabel={$q$},
    xtick={0.1,0.2,0.3,0.4,0.5},ytick={0.02,0.04,0.06,0.08,0.10,0.12},
    scaled ticks=false,
    yticklabel style={/pgf/number format/fixed,/pgf/number format/precision=2,
      /pgf/number format/fixed zerofill},
    title={(b) Dominant local mechanism},title style={font=\small,yshift=17pt},
    label style={font=\small},tick label style={font=\scriptsize},
    axis on top,tick align=outside]
\path[fill=gapStable!14] (axis cs:0.02,0.005) rectangle (axis cs:0.5,0.13);
\addplot[draw=none,fill=gapThreshold!16] coordinates {
  (0.02,0.003678794) (0.234161947,0.043071683) (0.234161947,0.13) (0.02,0.13)
} \closedcycle;
\addplot[draw=none,fill=gapBoundary!13] coordinates {(0.234161947,0.043071683) (0.237959634,0.043208365) (0.241757320,0.043345073) (0.245555007,0.043481807) (0.249352693,0.043618566) (0.253150380,0.043755352) (0.256948066,0.043892164) (0.260745753,0.044029001) (0.264543439,0.044165865) (0.268341126,0.044302755) (0.272138812,0.044439670) (0.275936498,0.044576612) (0.279734185,0.044713580) (0.283531871,0.044850573) (0.287329558,0.044987593) (0.291127244,0.045124639) (0.294924931,0.045261710) (0.298722617,0.045398808) (0.302520304,0.045535932) (0.306317990,0.045673082) (0.310115677,0.045810258) (0.313913363,0.045947460) (0.317711050,0.046084688) (0.321508736,0.046221942) (0.325306423,0.046359222) (0.329104109,0.046496528) (0.332901795,0.046633860) (0.336699482,0.046771218) (0.340497168,0.046908603) (0.344294855,0.047046013) (0.348092541,0.047183450) (0.351890228,0.047320913) (0.355687914,0.047458401) (0.359485601,0.047595916) (0.363283287,0.047733457) (0.367080974,0.047871024) (0.370878660,0.048008618) (0.374676347,0.048146237) (0.378474033,0.048283883) (0.382271720,0.048421554) (0.386069406,0.048559252) (0.389867092,0.048696976) (0.393664779,0.048834726) (0.397462465,0.048972502) (0.401260152,0.049110305) (0.405057838,0.049248133) (0.408855525,0.049385988) (0.412653211,0.049523869) (0.416450898,0.049661776) (0.420248584,0.049799710) (0.424046271,0.049937669) (0.427843957,0.050075655) (0.431641644,0.050213667) (0.435439330,0.050351705) (0.439237017,0.050489769) (0.443034703,0.050627860) (0.446832389,0.050765977) (0.450630076,0.050904120) (0.454427762,0.051042289) (0.458225449,0.051180484) (0.462023135,0.051318706) (0.465820822,0.051456954) (0.469618508,0.051595228) (0.473416195,0.051733529) (0.477213881,0.051871856) (0.481011568,0.052010209) (0.484809254,0.052148588) (0.488606941,0.052286994) (0.492404627,0.052425425) (0.496202314,0.052563884) (0.500000000,0.052702368) (0.500000000,0.130000000) (0.234161947,0.130000000)} \closedcycle;
\addplot[black!70,thick,domain=0.02:0.234161947364,samples=2] {0.183939720586*x};
\addplot[black!70,thick,domain=0.234161947364:0.5,samples=81]
  {(exp(0.05*(1+x))-1)/(0.2*exp(2))};
\draw[black!70,thick] (axis cs:0.234161947,0.043071683)
  --(axis cs:0.234161947,0.13);
\node[font=\scriptsize,align=center] at (axis cs:0.115,0.065)
  {Threshold\\$\gamma_*=\gamma_1$};
\node[font=\scriptsize,align=center] at (axis cs:0.37,0.072)
  {Boundary\\$\gamma_*=\gamma_0$};
\node[font=\scriptsize,align=center] at (axis cs:0.21,0.017)
  {Stable\\$\gamma_*=\gamma_2$};
\addplot[only marks,mark=*,mark size=1.8pt,black]
  coordinates {(0.4,0.1) (0.1,0.1) (0.4,0.02)};
\node[anchor=south,font=\small] at (axis cs:0.4,0.103) {$A$};
\node[anchor=south,font=\small] at (axis cs:0.1,0.103) {$B$};
\node[anchor=south,font=\small] at (axis cs:0.4,0.023) {$C$};
\end{axis}
\end{tikzpicture}
\caption{The limiting principal spectral gap in the explicit family with $a=1/10$, $b=1/5$,
and $M=10$.
(a) Local rates at the parameter choices $A,B,C$. Stars mark
the minimum, and the displayed values of $\gamma_*$ are rounded
to four decimal places.
(b) Regions in which the boundary, threshold, or stable-equilibrium
rate is smallest, for $0.02\le p\le0.5$ and
$0.005\le q\le0.13$. Black curves indicate equality of the
two smallest rates.
The values follow from the explicit formulas for the local rates.}
\label{fig:three-mechanisms}
\end{figure}
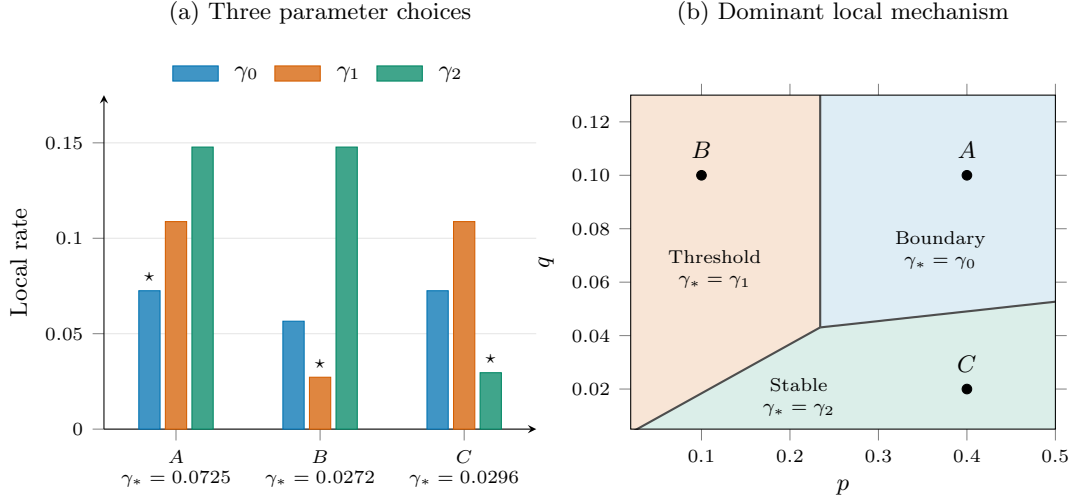

\subsection{Diffusion approximation and further spectral questions}
\label{subsec:diffusion-comparison}
The three-rate formula gives a natural point of comparison with a
diffusion approximation. The issue is whether an approximation on
the density scale also captures the boundary contribution to the
spectrum.
Let
$\mathcal G_K$ denote the generator of the scaled density process $X^K/K$.
For each fixed $f\in C_c^3((0,\infty))$, Taylor expansion yields
\[
(\mathcal G_Kf)(x)=(\mathcal A_Kf)(x)+O_{f,J}(K^{-2}),
\]
uniformly over lattice points $x=n/K$ belonging to any compact interval
$J\subset(0,\infty)$, where
\[
(\mathcal A_Kf)(x):=V(x)f'(x)
+\frac{x(\widetilde\lambda(x)+\widetilde\mu(x))}{2K}f''(x).
\]
The operator $\mathcal A_K$ is the formal generator for the absorbing diffusion
\[
dY_t^K=V(Y_t^K)\,dt+
\sqrt{\frac{Y_t^K\bigl(\widetilde\lambda(Y_t^K)
+\widetilde\mu(Y_t^K)\bigr)}{K}}\,dW_t.
\]
Here $(W_t)_{t\ge0}$ is standard one-dimensional Brownian
motion. The coefficients have the form considered in
\cite[Eq.~(1.1)]{Yan-Zhou-continuous}, with
\[
b(x)=V(x),\qquad
 a(x)=x\bigl(\widetilde\lambda(x)+\widetilde\mu(x)\bigr),\qquad
 \epsilon=K^{-1/2}.
\]
In particular, $b'(0)=-\gamma_0<0$ and
$a'(0)=\widetilde\lambda(0)+\widetilde\mu(0)>0$.
The approximating diffusion has an attracting origin, and its
infinitesimal variance vanishes linearly at zero. Applying results
from Yan and Zhou requires their regularity and tail assumptions,
including $a\in C^3((0,\infty))$. This regularity does not follow
from {\rm(A1)}, which assumes only $C^2$ smoothness.

For $K$-dependent test functions resolving the microscopic boundary
layer, the remainder in the density-scale expansion need not tend to
zero. For example, take a smooth, compactly
supported function $g$ on $[0,\infty)$ with $g(0)=0$, and define
$f_K(x)=g(Kx)$. At each fixed index $n\ge1$, direct calculation gives
\[
\begin{aligned}
(\mathcal G_K f_K)(n/K)&\longrightarrow
 (\mathcal L_b^{(0)}g)_n,\\
(\mathcal A_K f_K)(n/K)&\longrightarrow
 (\mathcal B g)(n),
\end{aligned}
\qquad
\mathcal B=-\gamma_0y\partial_y+
\frac{\alpha_0+\beta_0}{2}\,y\partial_{yy}.
\]
Here $\alpha_0,\beta_0$ are the boundary-scale constants from
Subsection~\ref{subsec:local-bounds}, and $\mathcal L_b^{(0)}$ is the frozen-boundary generator appearing in
\resultref{Lemma~}{lem:left-boundary-coercive}. In the first limit we identify the continuous function $g$ with the sequence $g_n=g(n)$.

The frozen discrete chain has unit jump size, whereas $\mathcal B$ depends only on the first two jump moments. In particular, choosing $g(y)=y^3$ on $[0,3]$ and evaluating at $n=1$ gives
\[
(\mathcal L_b^{(0)}g)_1-(\mathcal B g)(1)
=(7\alpha_0-\beta_0)-6\alpha_0=-\gamma_0\ne0.
\]
Thus the truncation error need not vanish for profiles on the scale
$x=O(K^{-1})$: the density-scale expansion does not give a uniform
approximation of the microscopic generator near the absorbing
boundary. This example uses $K$-dependent test functions and does not
by itself rule out matching limits for quasi-stationary distributions,
principal eigenvalues, or spectral gaps.

In fact, the two frozen absorbing models have the same spectral
bottom. If $\mathcal H_b^{\mathrm{diff}}$ denotes the absorbing
self-adjoint realization of $-\mathcal B$ defined in
Appendix~\ref{sec:frozen-diffusion-bottom}, then
\resultref{Proposition~}{prop:frozen-diffusion-bottom} gives
\[
\inf\operatorname{Spec}(\mathcal H_b^{\mathrm{diff}})
=\inf\operatorname{Spec}(\mathcal H_b^{(0)})=\gamma_0.
\]
The proof uses the positive eigenfunction $\ell(y)=y$ and a
ground-state identity; the form-domain justification is given in
that appendix. Thus the boundary rate $\gamma_0$ is present in both
local models. Comparing their global gap limits would also require
the interior estimates, tail control, and the corresponding
orthogonality argument for the diffusion. For the discrete process,
these are supplied directly by the IMS proof of
\resultref{Theorem~}{thm:spectral-gap-limit}.

Doering et~al.~\cite[Section~3.2]{Doering-Sargsyan-Sander2005} give birth--death
examples in which the standard Fokker--Planck approximation fails to
reproduce the leading exponential asymptotics of extinction times.
Their results concern extinction times and do not settle the
spectral-gap comparison discussed here.

Within the discrete model, the three possible minimizing rates lead
to a further question about the second eigenvector of the Jacobi
operator. When the minimum is unique, the local energy bounds suggest
concentration in the corresponding region. Identifying its limiting
profile requires compactness on the relevant local scale and control
of the associated eigenspace. The present proof provides interior
compactness; boundary compactness and a complete treatment of the
limiting eigenvectors are left open. When local rates coincide, one
must also determine how the eigenvectors combine modes from the
different regions.

The local spectra suggest an extension of the two-spectrum limit
in \cite[Theorem~1.2]{Chazottes-Collet-Meleard2023} to the
strong-Allee setting. Here the three local spectra are
\[
\begin{array}{ll}
\text{frozen boundary:} & \{k\gamma_0:k\ge1\},\\
\text{unstable threshold:} & \{k\gamma_1:k\ge1\},\\
\text{stable equilibrium:} & \{k\gamma_2:k\ge0\}.
\end{array}
\]
The boundary spectrum is recalled in
Appendix~\ref{sec:frozen-boundary-spectrum}. The stable oscillator has
a zero mode, consistent with the exponentially small global principal
eigenvalue. Define the multiset union
\[
\mathcal M:=\biguplus_{i=0}^{2}\{k\gamma_i:k=1,2,\ldots\},
\]
where coinciding local levels are counted with multiplicity. Write the
elements in nondecreasing order as
$\Lambda_2\le\Lambda_3\le\cdots$, and set $\Lambda_1=0$. We conjecture that for every fixed integer $j\ge1$,
\[
\lim_{K\to\infty}\rho_j(K)=\Lambda_j.
\]
Equation~\eqref{eq:rho1-vanishes} and
\resultref{Theorem~}{thm:spectral-gap-limit} verify this conjecture for $j=1$ and $j=2$.
A proof for higher eigenvalues would require finite-dimensional
min--max comparisons and a description of the local spectral
subspaces. The coercivity bounds in
\resultref{Lemma~}{lem:middle-region-coercivity} and
\resultref{Lemma~}{lem:right-tail-coercive} control the intermediate
region and the far tail. One would still need boundary compactness
and a treatment of multiplicities when different local operators
share an eigenvalue. The conjecture therefore goes beyond the
principal-gap limit proved here.

\appendix
\section{Spectrum of the frozen boundary operator}
\label{sec:frozen-boundary-spectrum}

The full frozen spectrum follows from
\cite[Theorem~7.1]{Chazottes-Collet-Meleard2023}.
On $c_{00}(\mathbb N^*)$, the Jacobi action is
\[
(\mathcal H_b^{(0)}z)_n
=n(\alpha_0+\beta_0)z_n
-\sqrt{\alpha_0\beta_0n(n+1)}\,z_{n+1}
-\sqrt{\alpha_0\beta_0n(n-1)}\,z_{n-1},
\]
where the last term vanishes at $n=1$.
The coefficients depend only on the sum and product of
$\alpha_0,\beta_0$. In the notation of the cited theorem, taking
$b'(0)=\beta_0$ and $d'(0)=\alpha_0$ identifies this operator with $-M_0$.
That theorem gives a self-adjoint closure on $\ell^2$, which
coincides with the Friedrichs realization used here. Consequently,
\begin{equation}\label{eq:frozen-full-spectrum}
\operatorname{Spec}(\mathcal H_b^{(0)})
=\{k\gamma_0:k=1,2,\ldots\},
\end{equation}
and every eigenvalue is simple. The lowest eigenvector in
\cite[Eq.~(7.1)]{Chazottes-Collet-Meleard2023} is proportional to
$\sqrt n\,r_0^{n/2}$, in agreement with the ground state used in
\resultref{Lemma~}{lem:left-boundary-coercive}.

\section{Spectral bottom of the frozen boundary diffusion}
\label{sec:frozen-diffusion-bottom}

We give the form-domain argument used in
Subsection~\ref{subsec:diffusion-comparison}. Recall that
$0<\alpha_0<\beta_0$ and $\gamma_0=\beta_0-\alpha_0$. Set
\[
c_b:=\frac{\alpha_0+\beta_0}{2},\qquad
\kappa_b:=\frac{\gamma_0}{c_b},\qquad
m_b(dy):=\frac{e^{-\kappa_b y}}{c_b y}\,dy.
\]
On $L^2(m_b)$, consider the densely defined symmetric form
\[
\mathcal E_b(u,v):=\int_0^\infty
e^{-\kappa_b y}u'(y)v'(y)\,dy,
\qquad u,v\in C_c^\infty((0,\infty)).
\]

\begin{proposition}\label{prop:frozen-diffusion-bottom}
The form $\mathcal E_b$ is closable. Denote its closure by the same
symbol, and let $\mathcal H_b^{\mathrm{diff}}$ be the associated
nonnegative self-adjoint operator, the absorbing realization of
\[
-\mathcal B=\gamma_0y\partial_y-c_by\partial_{yy}.
\]
Then $\ell(y)=y$ belongs to $D(\mathcal H_b^{\mathrm{diff}})$,
$\mathcal H_b^{\mathrm{diff}}\ell=\gamma_0\ell$, and
\[
\inf\operatorname{Spec}(\mathcal H_b^{\mathrm{diff}})
=\gamma_0.
\]
\end{proposition}

\begin{proof}
Suppose $u_j\to0$ in $L^2(m_b)$ and $u_j'$ converges in
$L^2(e^{-\kappa_b y}dy)$. On every compact subinterval of
$(0,\infty)$, both weights are bounded above and below by positive
constants. The limit of the derivatives is therefore zero in the
sense of distributions, and hence almost everywhere. This proves
closability.

For $u=yh\in C_c^\infty((0,\infty))$, integration by parts gives
\[
\mathcal E_b(u,u)-\gamma_0\|u\|_{L^2(m_b)}^2
=\int_0^\infty y^2e^{-\kappa_b y}|h'(y)|^2\,dy\ge0.
\]
The lower bound extends to the full form domain by closure.

To verify that $\ell(y)=y$ belongs to $D(\mathcal E_b)$, fix
$\xi\in C^\infty(\mathbb R)$ with $0\le\xi\le1$, $\xi=0$ on
$(-\infty,1]$, $\xi=1$ on $[2,\infty)$, and
$\|\xi'\|_\infty\le2$. For $0<\varepsilon<1$ and $R>2$, set
\[
\ell_{\varepsilon,R}(y)
:=y\xi(y/\varepsilon)\bigl(1-\xi(y/R)\bigr).
\]
This function lies in $C_c^\infty((0,\infty))$ and agrees with
$\ell$ on $[2\varepsilon,R]$. Direct integration gives
\[
\|\ell_{\varepsilon,R}-\ell\|_{L^2(m_b)}^2
+\int_0^\infty e^{-\kappa_b y}
 |\ell_{\varepsilon,R}'(y)-1|^2\,dy
=O\bigl(\varepsilon+(1+R)e^{-\kappa_b R}\bigr),
\]
where the $O$ estimate refers to $\varepsilon\downarrow0$ and
$R\to\infty$, uniformly in both parameters. Its implicit constant
depends only on $\alpha_0,\beta_0$ and the fixed profile $\xi$.
These limits give convergence to
$\ell$ in $L^2(m_b)$ and show that these approximations are Cauchy
in form norm. By closedness, $\ell\in D(\mathcal E_b)$.
For every $v\in C_c^\infty((0,\infty))$, integration by parts yields
\[
\mathcal E_b(\ell,v)
=\kappa_b\int_0^\infty e^{-\kappa_b y}v(y)\,dy
=\gamma_0\langle\ell,v\rangle_{L^2(m_b)}.
\]
By density, this identity holds for every $v\in D(\mathcal E_b)$.
The representation theorem for closed forms therefore gives
$\ell\in D(\mathcal H_b^{\mathrm{diff}})$ and
$\mathcal H_b^{\mathrm{diff}}\ell=\gamma_0\ell$.
Together with the form lower bound, this identifies the spectral
bottom as $\gamma_0$.
\end{proof}

\end{document}